\documentclass[11pt,reqno]{amsproc}
\usepackage[colorlinks,linkcolor=Green,citecolor=blue]{hyperref}
\usepackage{lineno}
\modulolinenumbers[5]
\usepackage{times}
\usepackage{amsmath}
\usepackage{amssymb}
\usepackage{mathrsfs}
\usepackage{amsthm}
\usepackage{bm}
\usepackage{graphicx}
\usepackage{algorithm} 
\usepackage{algpseudocode}
\usepackage{color}
\definecolor{Green}{RGB}{0,128,0}
\usepackage{subfigure}
\usepackage{enumitem}
\usepackage{float}
\usepackage{geometry}
\newtheorem{Def}{Definition}[section]
\newtheorem{lemma}[Def]{Lemma}
\newtheorem{hypothesis}[Def]{Hypothesis}
\newtheorem{corollary}[Def]{Corollary}
\newtheorem{theorem}[Def]{Theorem}
\newtheorem{example}[Def]{Example}
\newtheorem{proposition}[Def]{Proposition}
\newtheorem{remark}[Def]{Remark}

\usepackage{color,xcolor}
\definecolor{Green}{RGB}{46 139 87} % added for Color
\usepackage{tabularray}
\UseTblrLibrary{diagbox}

\newcommand{\ud}{\mathrm d}
\newcommand{\R}{\mathbb{R}}
\numberwithin{equation}{section}

\newcommand{\E}{\mathbb{E}}\allowdisplaybreaks[4]

\begin{document}

\title[Dynamic domain semi-Lagrangian method]{A dynamic domain semi-Lagrangian method for\\ stochastic Vlasov equations}

\subjclass[2010]{35Q83, 60H15, 60H35, 65C30, 65J08.}
\author{Jianbo Cui}%\corref{A1}}
\address{Department of Applied Mathematics, The Hong Kong Polytechnic University, Hung Hom, Kowloon, HongKong}
%\fntext[A1]{Academy of Mathematics and Systems Science, Chinese Academy of Sciences, Beijing, China}
\email{jianbo.cui@polyu.edu.hk}

\author{Derui Sheng}%\corref{A1}}
\address{Department of Applied Mathematics, The Hong Kong Polytechnic University, Hung Hom, Kowloon, HongKong}
%\fntext[A2]{Academy of Mathematics and Systems Science, Chinese Academy of Sciences, Beijing, China}
\email{sdr@lsec.cc.ac.cn (Corresponding author)}

\author{Chenhui Zhang}%\corref{A1}}
\address{Department of Applied Mathematics, The Hong Kong Polytechnic University, Hung Hom, Kowloon, HongKong }
\email{czhang9@163.com}
%\cortext[A1]{Academy of Mathematics and Systems Science, Chinese Academy of Sciences, Beijing
%100190, China; School of Mathematical Sciences, University of Chinese Academy of
%Sciences, Beijing 100049, China
\author{Tau Zhou}%\corref{A1}}
\address{Key Laboratory of Computing and Stochastic Mathematics (Ministry of Education), School of Mathematics and Statistics, Hunan Normal University, Changsha, Hunan 410081, P. R. China}
\email{zt@hunnu.edu.cn}

\thanks{This work is supported by the Hong Kong Research Grant Council GRF grant 15302823, NSFC/RGC Joint Research Scheme N-PolyU5141/24, NSFC grant 12301526, internal funds (P0039016, P0045336) from Hong Kong Polytechnic University, MOST National Key R\&D Program No. 2024FA1015900, and the
CAS AMSS-PolyU Joint Laboratory of Applied Mathematics.
}

\keywords{Dynamic domain adaptation strategy, semi-Lagrangian method, stochastic Vlasov--Poisson equation, transport noise, volume-preserving integrator}

\begin{abstract}
We propose a dynamic domain semi-Lagrangian method for stochastic Vlasov equations driven by transport noises, which arise in plasma physics and astrophysics. 
This method combines the volume-preserving property of stochastic characteristics with a dynamic domain adaptation strategy and a reconstruction procedure.
It offers a substantial reduction in computational costs compared to the traditional semi-Lagrangian techniques for stochastic problems.
Furthermore, we present the first-order convergence analysis of the proposed method, partially addressing the conjecture in \cite{BC24} on the convergence order of numerical methods for stochastic Vlasov equations.
Several numerical tests are provided to show good performance of the proposed method.
\end{abstract}

\maketitle
%\tableofcontents
\section{Introduction}
In astrophysics and plasma physics, the kinetic equation is a fundamental framework for describing collisionless plasmas, which models the evolution of charged particles within an electromagnetic field \cite{EO14vp,FP84}. 
Beyond magnetic and electric effects, particles are also affected by random forces, which often arise from the thermal fluctuation, turbulence, or external noise 
 (see e.g., \cite{BB21,GIVW10}). In such scenarios, the distribution of particles is governed by the stochastic Vlasov equation with transport noise (see, e.g., \cite{BC24, FFPV17})
\begin{align}\label{eq:Vla}
	&\ud_tf +\left(v\cdot \nabla_xf+E(t,x)\cdot\nabla_vf\right)\ud t
	+\sum_{k=1}^K\sigma_k(x)\cdot\nabla_v f\odot\ud \beta_k(t)=0
\end{align}
where the time $t\ge 0$, the position $x\in\mathbb{T}^d$ ($d$-dimensional torus) and the velocity $v\in\R^d$. Here, $\beta_k$'s are independent Brownian motions, $\sigma_k$'s are vector fields on $\mathbb{T}^d$, $\odot$ denotes the Stratonovich product, and $\cdot$ denotes the inner product in $\R^d$ (see section \ref{S2} for more details).
When the magnetic effect are neglected, $E:[0,T]\times\mathbb{T}^d\to \R^d$ represents the electric field. It can either be 
externally imposed (i.e., $E$ is independent of $f$) or determined self-consistently via the Poisson equation \cite{GR96}
 \begin{equation}\label{eq:VPE}
E(t,x)=-\nabla u(t,x),\quad -
\Delta u(t,x)=\int_{\R^d}f(t,x,v)\ud v-1, 
\end{equation}
where $u$ is the electric potential. Physically, the stochastic Vlasov--Poisson equation \eqref{eq:Vla}-\eqref{eq:VPE} describes charged particles in a turbulent regime subject to a non-self-consistent stochastic electric field \cite{BB21}. It has been reported in \cite{DFV14} that the transport noise ‘$\sum_{k=1}^K\sigma_k(x)\cdot\nabla_v f\odot \ud \beta_k(t)$’ has a regularizing effect, preventing the collapse of the particle system. 

The absence of analytic solutions for Vlasov-type equations, especially in the nonlinear regime, has driven extensive research into their numerical study.
A widely used approach is the semi-Lagrangian method (also known as the Eulerian method), which propagates the solution along characteristics  and uses a reconstruction procedure to recover the solution from grid points in phase space at each time step. Popular reconstruction techniques include the interpolation \cite{BN04, CK76}, finite difference methods \cite{QS11FDM}, and discontinuous Galerkin methods \cite{EO14, QS11}. Another prevalent approach is the particle-in-cell method, which approximates the continuous density using a weighted Klimontovich representation (see, e.g., \cite{LY23, LCHPS24, VA91}). Among these studies, numerically preserving conservative physical quantities is important for enhancing stability and reliability in long time simulations (see, e.g., \cite{EJ21}).

Despite fruitful results in deterministic settings, the scientific computing and numerical analysis of stochastic Vlasov equations are far from well-understood. To the best of our knowledge, the only existing work is \cite{BC24}, which introduces a Lie--Trotter splitting integrator as a temporal semi-discretization for linear Vlasov equations with stochastic perturbations. The authors conjecture in \cite{BC24} that this splitting scheme achieves first-order convergence in the mean-square sense. For more complex nonlinear problems, a full discretization is essential and warrants further investigation. These considerations are the primary motivations for our current work. 

To propose an effective full discretization for \eqref{eq:Vla}, we encounter several challenges. 
First, we cannot directly apply  the traditional semi-Lagrangian method (see, e.g., \cite{BN04, EO14}) that truncates the velocity space into a preset bounded domain.
Indeed, stochastic perturbations in \eqref{eq:Vla}
cause the
particle acceleration to behave like white noise, which is not uniformly bounded over time, as reflected in the stochastic characteristics (see, e.g., \cite{KH19}):
\begin{equation}\label{eq:SDE}
\left\{
\begin{split}
\ud X_t&=V_t \ud t,\\
\ud V_t&=E(t,X_t) \ud t+\sum_{k=1}^K\sigma_k(X_t)\ud \beta_k(t).
\end{split}
\right.
\end{equation}
Second, 
the \emph{a 
priori} estimate of \eqref{eq:SDE} suggests that the diameter of the support of the full discretization
 for \eqref{eq:Vla} 
 is nearly proportional to $\tau^{-\frac12}$ with $\tau>0$ being the time stepsize (see Remark \ref{rem:1}). This will lead to 
an expensive computational cost 
since the
 computational domain may expand as the diameter of the support increases.
 Third,
the transport noise 
causes the averaged total energy to increase over time, while the mass and averaged momentum remain invariant for the stochastic Vlasov--Poisson equation (see section \ref{S2}). Numerically
capturing these physical features presents another challenge, since standard discretizations, 
such as the full discretization related to the Euler--Maruyama method of  \eqref{eq:SDE},
 may not accurately approximate these evolution laws ({see Fig. \ref{Fig:Norm} for a comparison)}.

To address these challenges, we propose a novel numerical method that combines the semi-Lagrangian approach with a dynamic domain adaptation strategy.  
We first truncate the velocity domain into a bounded one based on a preset threshold $\epsilon_0$, and adaptively update the computational velocity domain at each time iteration.
This approach significantly reduces the computational cost and enhances the efficiency of the traditional semi-Lagragian method for stochastic problems (see Tab.\ \ref{Tabletime} for a comparison). 
Then in the dynamic domain in phase space, we  use volume-preserving integrators to propagate the inverse flow of the stochastic characteristics \eqref{eq:SDE},
which effectively reduces errors in the integral preservation (see section \ref{S5.3}). Meanwhile, the reconstruction procedure is  implemented using techniques such as a positivity-preserving Lagrange first-order interpolation. Consequently, the proposed method performs well in simulating the evolution of physical quantities, such as the mass, averaged momentum, and averaged energy (see section \ref{S5.4}).

In Theorem \ref{thm:DSLC}, we further present the convergence analysis of the proposed method, which reveals conditions on the time stepsize $\tau$, threshold $\epsilon_0$, and spatial stepsizes to ensure the stability and accuracy. Achieving first-order convergence in time hinges on the introduction of an auxiliary process as an intermediate, as traditional approaches typically yield only half-order convergence in time (see Remark \ref{rem:secconvergence}). 
This gives a positive answer to the conjecture in \cite{BC24} regarding  the temporal convergence rate of numerical schemes for stochastic Vlasov equations, particularly in the case of transport noise.

The rest of this paper is organized as follows. Section \ref{S2} introduces the necessary notations and examines the evolution of physical quantities associated with the stochastic Vlasov equation. We propose the dynamic domain semi-Lagrangian method (i.e., Algorithm \ref{Algo:2}) in section \ref{S3} and then present its convergence analysis in section \ref{S4}. 
Finally, in section \ref{S:NE}, several numerical experiments are provided to validate the accuracy and efficiency of Algorithm \ref{Algo:2}, as well as the effect of transport noise on the evolution of physical quantities.

\section{Preliminaries}\label{S2}
In this section, we present the preliminaries and the evolution law of physical quantities associated with \eqref{eq:Vla}. Let us begin with some notations.
For $p\in[1,\infty)$, let $L_{x,v}^{p}$ be the space of all $p$th integrable functions $g:\mathbb{T}^d\times\R^d\to\R$ equipped with the usual norm 
\begin{align*}
\|g\|_{L_{x,v}^{p}}:=\left(\int_{\mathbb{T}^d\times\R^d}|g(x,v)|^p\ud x\ud v\right)^{\frac1p}
\end{align*}
where $\mathbb{T}^d:=(\R/L\mathbb{Z})^d$ is a $d$-dimensional torus with some positive constant $L$.
Denote by $L_{x,v}^{\infty}$ the space of all measurable functions $g:\mathbb{T}^d\times\R^d\to \R$ with a finite essential supremum endowed with the norm $\|g\|_{L_{x, v}^{\infty}}:=\textup{ess\,sup}_{(x, v) \in \mathbb{T}^d \times \mathbb{R}^d}
|g(x, v)|.$ Let $K$ be an positive integer and $\beta:=(\beta_1,\ldots,\beta_K)^\top$ a $K$-dimensional standard Brownian motion on a complete probability space $(\Omega,\mathscr{F},\mathbb{P})$.
 Suppose that the initial value $f_0$ of \eqref{eq:Vla} is a non-negative and globally Lipschitz continous function with  
$\|f_0\|_{L_{x,v}^1}=1$. In the sequel, we impose the following conditions.

\begin{hypothesis}\label{Asp:1}
Assume the following conditions (A1) and (A2) or (A2)$^\prime$:
\begin{enumerate}
\item[(A1)\phantom{$^\prime$}] $\sigma:=(\sigma_1,\ldots,\sigma_K)^\top$ is globally Lipschitz continuous;
 \item[(A2)\phantom{$^\prime$}] $E:[0,T]\times\mathbb{T}^d\to \R^d$ is a Lipschitz continuous function (Case (I)); 
 \item[(A2)$^\prime$] $E$ is the self-contained electric field given by \eqref{eq:VPE} (Case (II)),
 \end{enumerate}
 such that the solution of \eqref{eq:Vla} exists in the following sense:
 \begin{equation*}
f(t,x,v)=f_0(\widehat{\phi}_{0,t}(x,v)),\quad t>0, (x,v)\in\mathbb{T}^d\times\R^d,\quad \textup{a.s.}
\end{equation*}

Here, $\widehat{\phi}_{0,t}$ denotes the inverse map of $\phi_{0,t}$ with
$\{\phi_{0,t}\}_{t\ge 0}$ being the stochastic flow of homeomorphisms associated with \eqref{eq:SDE} (see, e.g., \cite{KH19}). Namely, $\phi_{0,t}(x,v)$ is the solution of \eqref{eq:SDE} at $t$ starting from time $0$ with the initial value $(x,v)$.
\end{hypothesis}

Under (A2) in Hypothesis \ref{Asp:1},
the model \eqref{eq:Vla} reduces to the stochastic linear Vlasov equation (Case (I)) driven by transport noise (see \cite{BC24}). 
Hypothesis \ref{Asp:1}(A2)$^\prime$ corresponds to the stochastic Vlasov--Poisson equation (Case (II)). We denote by $G$ the Green function for the negative Laplace–Beltrami operator on $\mathbb T^d$
, i.e., $-\Delta G=\boldsymbol{\delta}_0-1$, where $\boldsymbol{\delta}_0$ is the Dirac delta function centered at the origin. Introducing the Coulomb kernel $H=-\nabla_xG$,
it follows from \eqref{eq:VPE} that 
\begin{equation*}
E(t,x)=\int_{\mathbb{T}^d} H(x-y)\rho_f(t,y)\ud y
\end{equation*}
where $\rho_f(t,y):=\int_{\R^d}f(t,y,v)\ud v$ denotes the mass density. 
For a detailed study on the well-posedness of stochastic Vlasov equations, we refer to \cite{BC24,DFV14} and references therein. We also refer to \cite{CLZ20,CLZ24} for 
insights into the connection between stochastic Vlasov equations and stochastic Wasserstein Hamiltonian flows. 

\begin{remark}\label{rem1}
When $\sigma$ is a constant matrix,
the solution of \eqref{eq:Vla}-\eqref{eq:VPE} is equivalent to the solution of the deterministic Vlasov--Poisson equation under a stochastic coordinate transformation  (see also \cite{DFV14}). To illustrate this, we denote by $(E^{0},f^{0})$ the solution of \eqref{eq:Vla}-\eqref{eq:VPE} with $\sigma\equiv0$.
By the It\^o formula, it can be verified that for any constant matrix $\tilde{\sigma}\in\R^{d\times K}$, the pair $(E^{\tilde{\sigma}},f^{\tilde{\sigma}})$ given by
\begin{gather*}
E^{\tilde{\sigma}}(t,x):=E^{0}\Big(t,x-\tilde{\sigma}\int_0^t \beta(s)\ud s\Big),\\
f^{\tilde{\sigma}}(t,x,v):=f^0\Big(t,x-\tilde{\sigma}\int_0^t \beta(s)\ud s,v-\tilde{\sigma} \beta(t)\Big)
\end{gather*}
satisfies the stochastic Vlasov--Poisson equation \eqref{eq:Vla}-\eqref{eq:VPE} with $\sigma\equiv\tilde\sigma$. 
\end{remark}

The momentum and kinetic energy associated with the Vlasov equation are 
\begin{gather*}
\mathcal{P}[f(t)]=\int_{\mathbb{T}^d\times\mathbb{R}^d}vf(t,x,v)\ud x\ud v\qquad\text{and}\qquad
\mathcal{K}[f(t)]=\frac12\int_{\mathbb{T}^d\times\mathbb{R}^d}\|v\|^2f(t,x,v)\ud x\ud v,
\end{gather*}
respectively, where $\|\cdot\|$ stands for the Euclidean norm in $\R^d$. The total energy reads
\begin{align}\label{eq:DefHF}
\mathcal H[f(t)]&:= \mathcal{K}[f(t)]\\
&\quad+\begin{cases}
\int_{\mathbb{T}^d\times\R^d}u(x)f(t,x,v)\ud x\ud v,\quad &\text{if } E=-\nabla u\text{ for some } u:\mathbb{T}^d\to\R,\\\notag
\frac{1}{2}\int_{\mathbb{T}^d}\|E(t,x)\|^2\ud x,\quad &\text{if } E \text{ satisfies } \eqref{eq:VPE}.
\end{cases}
\end{align}
It is known that the mass $\|f(t)\|_{L^1_{x,v}}$, the momentum $\mathcal{P}[f(t)]$,
and the total energy $\mathcal H[f(t)]$
are invariants for the determinisitic Vlasov--Poisson equation
(see, e.g., \cite{GHS22}). In the following, we show how the transport noise affects the evolution of relevant physical quantities.

As illustrated in \cite[Proposition 11]{BC24}, the solution of \eqref{eq:Vla} admits the following preservation properties.
\begin{enumerate}
\item[(1)] Preservation of positivity: $f(t, x, v) \ge 0$ almost surely for all $t\ge0$ and $(x, v) \in \mathbb{T}^d \times \mathbb{R}^d$.
\item[(2)] Preservation of integrals: Let $\Phi: \R\rightarrow\R$ be measurable mapping. Then for any $t\ge0$,
\begin{align}\label{eq:fPhi}
\iint_{\mathbb{T}^d\times\mathbb{R}^d} \Phi\left(f(t, x, v)\right) \ud x \ud v=\iint_{\mathbb{T}^d\times\mathbb{R}^d} \Phi\left(f_0(x, v)\right) \ud x \ud v\quad \textup{a.s}.
\end{align}
In particular, if $f_0 \in L_{x, v}^p$ for some
$p \in[1, \infty]$, then for any $t\ge0$,
$
\left\|f(t)\right\|_{L_{x, v}^p}=\left\|f_0\right\|_{L_{x, v}^p}
$ almost surely.
\end{enumerate}
It can be seen from \eqref{eq:fPhi} that the mass associated with the stochastic Vlasov equation \eqref{eq:Vla} remains invariant for almost all sample paths. 

The equivalent It\^o formulation of \eqref{eq:Vla} takes the form
\begin{align}\label{eq:Ito}
&\ud_tf +(v\cdot \nabla_xf+E(t,x)\cdot\nabla_vf)\ud t
+\sum_{k=1}^K\sigma_k(x)\cdot\nabla_v f\ud \beta_k(t)\\\notag
&=\frac{1}{2}\sum_{i,j=1}^d(\sigma\sigma^\top)_{i,j}(x)\partial_{v_i}\partial_{v_j}f\ud t.
\end{align}
By the It\^o formula and integrating by parts, it holds that
\begin{gather}\label{eq:dtPft}
\ud_t \mathcal{P}[f(t)]=\int_{\mathbb{T}^d} \rho_f E\ud x\ud t+\sum_{k=1}^K\int_{\mathbb{T}^d} \rho_f\sigma_k\ud x\ud\beta_k(t),\\\label{eq:dtKft}
\ud_t \mathcal{K}[f(t)]=\int_{\mathbb{T}^d}j_f\cdot E\ud x\ud t
+\sum_{k=1}^K\int_{\mathbb{T}^d}j_f\cdot\sigma_k\ud x\ud \beta_k(t)+\frac12\int_{\mathbb{T}^d}\textup{Tr}(\sigma\sigma^\top)\rho_f\ud x\ud t,
\end{gather}
where $\rho_f(t,x)=\int_{\R^d}f(t,x,v)\ud v$ is the mass density and $j_f(t,x):=\int_{\R^d}vf(t,x,v)\ud v$ denotes the momentum density. This shows that the transport noise breaks conservation laws of momentum and total energy. In the following, we study the evolution laws of the averaged momentum $\E[\mathcal{P}[f(t)]]$, averaged kinetic energy $\E[\mathcal{K}[f(t)]]$, and averaged total energy $\E[\mathcal{H}[f(t)]]$. Hereafter, $\E[\,\cdot\,]$ denotes the expectation with respect to $\mathbb{P}$.

\begin{proposition}\label{tho:Hft}
Let Hypothesis \ref{Asp:1} hold and let $f$ be the solution of \eqref{eq:Vla}.
\begin{enumerate}
\item[(i)]If $E\equiv E_0$ for some constant vector $E_0\in\R^d$, then the averaged momentum grows linearly with respect to time, i.e., for any $t\ge0$,
\begin{equation}\label{eq:momentum}
\E\left[\mathcal{P}[f(t)]\right]=\E\left[\mathcal{P}[f_0]\right]+E_0t.
\end{equation}
If in addition $\sigma$ is a constant matrix, then for any $t\ge0$,
\begin{equation}\label{eq:Kenergy}
\E\left[\mathcal{K}[f(t)]\right]=\E\left[\mathcal{K}[f_0]\right]+\left(E_0\cdot\E\left[\mathcal{P}[f_0]\right]+\frac{1}{2}\textup{Tr}(\sigma\sigma^\top)\right)t+\frac12\|E_0\|^2t^2.
\end{equation}

 \item[(ii)]
If $E=-\nabla u$ for a twice continuously differentiable function $u:\mathbb{T}^d\to\R$, then for any $t\ge0$,
\begin{equation}\label{eq:Hf}
\E\left[\mathcal H[f(t)]\right]=\E\left[\mathcal H[f_0]\right]+\frac12\E\int_0^t\int_{\mathbb{T}^d}\textup{Tr}((\sigma\sigma^\top)(x))
\rho_f(s,x)\ud x\ud s.
\end{equation}
In particular, if $\sigma$ is a constant matrix, then for any $t\ge0$,
\begin{equation}\label{eq:Hftadd}
\E\left[\mathcal H[f(t)]\right]=\E\left[\mathcal H[f_0]\right]+\frac{t}{2}\textup{Tr}(\sigma\sigma^\top).
\end{equation}
\item[(iii)] If $E$ satisfies \eqref{eq:VPE}, then $
\E[\mathcal{P}[f(t)]]=\E[\mathcal{P}[f_0]]
$ and
\eqref{eq:Hf} hold for any $t\ge0$.
\end{enumerate}
\end{proposition}
\begin{proof}

(1) Integrating \eqref{eq:Ito} with respect to $v$ yields the continuity equation
$
\partial_t\rho_f+\textup{div}j_f=0$. Hence, for any $s\ge0$,
\begin{equation}\label{eq:rhof}
\int_{\mathbb{T}^d} \rho_f(s,x)\ud x=\int_{\mathbb{T}^d} \rho_f(0,x)\ud x=1,
\end{equation}
which, together with \eqref{eq:dtPft} and \eqref{eq:dtKft}, implies the property (i).

(2) In virtue of \eqref{eq:Ito}, \eqref{eq:DefHF}, \eqref{eq:dtKft}, and $E=-\nabla u$, integrating by parts shows that
\begin{equation*}
\ud_t \mathcal H[f(t)]=\sum_{k=1}^K\int_{\mathbb{T}^d}j_f\cdot\sigma_k\ud x\ud \beta_k(t)+\frac12\int_{\mathbb{T}^d}\textup{Tr}(\sigma\sigma^\top)\rho_f\ud x\ud t,\quad t\ge0.
\end{equation*}
Furthermore, if $\sigma$ is a constant matrix, then \eqref{eq:Hftadd} is a consequence of \eqref{eq:rhof}.

(3)
Taking \eqref{eq:VPE} into account and integrating by parts, it holds that
$$\int_{\mathbb{T}^d} \rho_f E\ud x=\int_{\mathbb{T}^d}\left(\textup{div} E+1\right)E\ud x=\int_{\mathbb{T}^d}\textup{div}(\nabla u) \nabla u\ud x-\int_{\mathbb{T}^d} \nabla u\ud x=0.$$
This, in combination with \eqref{eq:dtPft}, ensures that $\E[\mathcal{P}[f(t)]]=\E[\mathcal{P}[f_0]]$ for any $t\ge0$.
By \eqref{eq:VPE}, $
\partial_t\rho_f+\textup{div}j_f=0$, and integrating by parts, we obtain
\begin{align*}
\frac{\ud}{\ud t}\int_{\mathbb{T}^d}|E|^2\ud x&=2\int_{\mathbb{T}^d}E\cdot\partial_t E\ud x=-2\int_{\mathbb{T}^d}\nabla u\cdot\partial_t E\ud x\\
&=2\int_{\mathbb{T}^d} u\partial_t (\textup{div}E)\ud x=2\int_{\mathbb{T}^d} u\partial_t \rho_f\ud x=-2\int_{\mathbb{T}^d} u\textup{div}j_f\ud x\\
&=2\int_{\mathbb{T}^d} \nabla u\cdot j_f\ud x=-2\int_{\mathbb{T}^d} E\cdot j_f\ud x,
\end{align*}
which, along with \eqref{eq:dtKft}, completes the proof of \eqref{eq:Hf}. 
\end{proof}

\section{Dynamic domain semi-Lagrangian method}\label{S3}
In this section, we first introduce the volume-preserving integrator for the stochastic characteristics \eqref{eq:SDE} in subsection \ref{S3.1}, followed by a reconstruction step in subsection \ref{S3.2}. By combining these techniques with a dynamic domain strategy in subsection \ref{S3.3}, we propose a novel semi-Lagrangian method, as shown in Algorithm \ref{Algo:2}.

\subsection{Volume-preserving integrator}\label{S3.1}
We begin with numerically solving the stochastic characteristics \eqref{eq:SDE} of \eqref{eq:Vla}. 
Given a terminal time $T>0$,
let $\tau=T/N$ ($N\in\mathbb{N}_+$) and $t_n:=n\tau$ for $n=0,1,\ldots,N$. Denote by $\{\Psi_{t_{n-1},t_n}\}_{n=1}^N$ the numerical flow generated by a volume-preserving integrator of \eqref{eq:SDE} (see, e.g., \cite[section VI.9]{HLW02}) such that $\{\Psi_{t_{n-1},t_n}\}_{n=1}^N$ is invertible and volume-preserving. Thus, for any $(x,v)\in\mathbb{T}^d\times\mathbb R^d$, we can find the starting point $(X,V)$ such that $\Psi_{t_{n-1},t_n}(X,V)=(x,v)$ and
	\begin{equation}\label{eq:VP}
	\det\left(\nabla_{x,v}\Psi_{t_{n-1},t_n}(x,v)\right)=1,\quad \forall~(x,v)\in\mathbb{T}^d\times\R^d. 
	\end{equation}
	Then the inverse map $\widehat{\Psi}_{t_{n-1},t_n}$ of $\Psi_{t_{n-1},t_n}$ can be regarded as an approximation of the inverse map $\widehat{\phi}_{t_{n-1},t_n}$ of $\phi_{t_{n-1},t_n}$. 

Next, we introduce a class of numerical integrators for \eqref{eq:SDE} based on the splitting technique, which possess 
 the volume-preserving property \eqref{eq:VP}. Recall that the infinitesimal generator of \eqref{eq:SDE} is 
$$\mathcal L:= E(t,x)\cdot \nabla_v+v\cdot\nabla_x+\frac12\sigma\sigma^\top(x):\nabla_v^2,$$
where $\sigma\sigma^\top(x):\nabla_v^2=\sum_{i,j=1}^d\big((\sigma\sigma)^\top\big)_{i,j}(x)\partial_{v_i}\partial_{v_j}$. 
To propose splitting integrators for \eqref{eq:SDE}, we denote by
\begin{equation*}
\mathcal L_{v}:= v\cdot\nabla_x,\qquad
\mathcal L_{x}(t):=\mathcal L_{x,D}(t)+\mathcal L_{x,S},\quad t\in[0,T]
\end{equation*}
with $\mathcal L_{x,D}(t):= E(t,x)\cdot \nabla_v$ and $\mathcal L_{x,S}:=\frac12(\sigma\sigma^\top)(x):\nabla_v^2$.
For an operator $\tilde{\mathcal{L}}\in\{\mathcal{L}_{x}(t),\mathcal{L}_{v},\mathcal{L}_{x,D}(t),\mathcal{L}_{x,S}\}$ with a fixed $t\in[0,T]$, denote by 
 $\{\Phi^{\tilde{\mathcal{L}}}_{0,s}\}_{s\ge0}$ the solution flow 
 with infinitesimal generator $\tilde{\mathcal{L}}$. 
 Let $\delta\beta_{n,k}:=\beta_k(t_{n})-\beta_k(t_{n-1})$ with $n=1,\ldots,N,$ be the increments of the Brownian motions $\beta_k$, $k=1,2,\ldots,K$. We present three concrete numerical integrators for \eqref{eq:SDE} as follows.

(1) 
 The one-step numerical integrator 
$\Psi_{t_{n-1},t_n}:=
\Phi^{\mathcal L_{v}}_{0,\tau}\circ
\Phi^{\mathcal L_{x}(t_{n-1})}_{0,\tau}$
corresponds to a symplectic Euler method for \eqref{eq:SDE}. It can be verified that
 the inverse map $\widehat{\Psi}_{t_{n-1},t_n}$ associated with this symplectic Euler method satisfies
\begin{equation}
\label{eq:inv-Phi}\tag{SEM}
		\left\{
		\begin{split}
			\widehat{\Psi}^1_{t_{n-1},t_n}(x,v)&=x-\tau v, \\
			\widehat{\Psi}^2_{t_{n-1},t_n}(x,v)&=v-\tau E(t_{n-1},\widehat{\Psi}^1_{t_{n-1},t_n}(x,v)) - \sigma(\widehat{\Psi}^1_{t_{n-1},t_n}(x,v)) \delta\beta_{n}.
		\end{split}
		\right.\\
\end{equation}
with $\delta\beta_{n}:=(\delta\beta_{n,1},\ldots,\delta\beta_{n,K})^\top$.
 Hereafter, $\widehat{\Psi}^1_{t_{n-1},t_n}(x,v)$ 
and $\widehat{\Psi}^2_{t_{n-1},t_n}(x,v)$ denote the components of $\widehat{\Psi}_{t_{n-1},t_n}(x,v)$ in the position and velocity coordinates, respectively.

(2) Based on the Lie--Trotter splitting technique,
we can construct another one-step numerical integrator 
$\Psi_{t_{n-1},t_n}:=\Phi^{\mathcal{L}_{x,D}(t_{n-1})}_{0,\tau}\circ
\Phi^{\mathcal L_{v}}_{0,\tau}\circ
\Phi^{\mathcal{L}_{x,S}}_{0,\tau}$ for \eqref{eq:SDE}.
 The inverse map $\widehat{\Psi}_{t_{n-1},t_n}$ associated with this Lie--Trotter splitting method is given by
\begin{equation}
\label{eq:inv-LTS}\tag{LTSM}
		\left\{
		\begin{split}
			\widehat{\Psi}^1_{t_{n-1},t_n}(x,v)&=x-\tau v+\tau^2E(t_{n-1},x),\\
			\widehat{\Psi}^2_{t_{n-1},t_n}(x,v)&=v-\tau E(t_{n-1},x)-\sigma\big(\widehat{\Psi}^1_{t_{n-1},t_n}(x,v)\big)\delta\beta_{n}.
		\end{split}
		\right.
\end{equation}

(3) Applying the Strang splitting technique, we obtain a one-step numerical integrator 
$\Psi_{t_{n-1},t_n}:=\Phi^{\mathcal L_{v}}_{0,\tau/2}\circ\Phi^{\mathcal{L}_{x}(t_{n-1})}_{0,\tau}\circ
\Phi^{\mathcal L_{v}}_{0,\tau/2}$ for \eqref{eq:SDE}.
The inverse map $\widehat{\Psi}_{t_{n-1},t_n}$ associated with this Strang splitting method fulfills
\begin{equation}
\label{eq:inv-SS}\tag{SSM}
		\left\{
		\begin{split}
			\widehat{\Psi}^1_{t_{n-1},t_n}(x,v)&=x-\frac{1}{2}\tau (v+\widehat{\Psi}^2_{t_{n-1},t_n}(x,v)),\\
			\widehat{\Psi}^2_{t_{n-1},t_n}(x,v)&=v-\tau E\big(t_{n-1},x-\frac{1}{2}\tau v\big)-\sigma\big(x-\frac{1}{2}\tau v\big)\delta\beta_{n}.
		\end{split}
		\right.
\end{equation}
	Since the solution flows with the infinitisimal generators $\mathcal{L}_{x}(t_{n-1})$, $\mathcal{L}_{v}$, $\mathcal{L}_{x,D}(t_{n-1})$, and $\mathcal{L}_{x,S}$ are volume-preserving, the numerical integrators \eqref{eq:inv-Phi}, \eqref{eq:inv-LTS}, and \eqref{eq:inv-SS} for the inverse flow of \eqref{eq:SDE} fulfill the volume-preserving property \eqref{eq:VP}.

For every $k,l\in\{0,1,\ldots,N\}$ with $k<l$, we denote by $\Psi_{t_k,t_l}:=\Psi_{t_{l-1},t_{l}}\circ\cdots\circ\Psi_{t_{k+1},t_{k+2}}\circ\Psi_{t_k,t_{k+1}}$ and $\widehat{\Psi}_{t_k,t_l}:=\widehat{\Psi}_{t_k,t_{k+1}}\circ\widehat{\Psi}_{t_{k+1},t_{k+2}}\circ\cdots\circ\widehat{\Psi}_{t_{l-1},t_{l}}$ the numerical flow and its inverse map associated with \eqref{eq:SDE}, respectively.
Based on the numerical integrator for \eqref{eq:SDE},
we have the following temporal semi-discretization for \eqref{eq:Vla}:
\begin{equation}\label{eq:semi-dis}
	f^\tau_n(x,v)=f_{n-1}^\tau(\widehat{\Psi}_{t_{n-1},t_n}(x,v)),\quad (x,v)\in\mathbb{T}^d\times\R^d
\end{equation}
for $n=1,\ldots,N$
with $f^\tau_0=f_0$. Note that $f^\tau_n(x,v)=f_0(\widehat{\Psi}_{0,t_n}(x,v))$ for any $n=0,1,\ldots,N$. The numerical method \eqref{eq:semi-dis} with \eqref{eq:inv-LTS} (called the Lie--Trotter splitting scheme) has been studied in \cite{BC24} for \eqref{eq:Vla} with a given function $E:\mathbb{T}^d\to \R$. 
Repeating the same reasoning as in \cite[Proposition 12]{BC24}, one has that \eqref{eq:semi-dis} with \eqref{eq:inv-Phi}, \eqref{eq:inv-LTS}, or \eqref{eq:inv-SS}
 preserves positivity and integrals thanks to the volume-preserving property \eqref{eq:VP}.

The temporal semi-discretization \eqref{eq:semi-dis} can be used for Case (I) at each specified point $(x,v)\in\mathbb{T}^d\times\R^d$. However, it is not implementable for Case (II) due to the absence of the analytical formulation of $E$. 
To illustrate this, we take the case $d=1$ as an example for simplicity. By imposing a zero-mean electrostatic condition
 $\int_0^L E(t,x)\ud x=0, t\ge0$,
 the electric field satisfies (see, e.g., \cite{BN04}),
\begin{equation*}
	E(t,x)=\int_{0}^L \tilde{K}(x,y)\left(\rho_f(t,y)-1\right)\ud y \text{\quad with
		\quad}\tilde{K}(x,y)=\begin{cases}\frac{y}{L}-1,& 0\le x<y,\\
		\frac{y}{L},& y<x\le L.
	\end{cases}
\end{equation*}
The positivity of $f(t,x,v)$ implies that for any $(t,x)\in[0,T]\times\mathbb{T}$,
$$|E(t,x)|\le \int_{0}^L\int_{\R}|\tilde{K}(x,y)|f(t,y,v)\ud v\ud y+\int_{0}^L|\tilde{K}(x,y)|\ud y\le 2L.$$
Therefore, the term $E(t_{n-1},\cdot)$, which appears in the computation of $\widehat{\Psi}_{t_{n-1},t_n},$ can be  approximated 
by
\begin{equation}\label{eq:Electric}
\widetilde{E}_{n-1}(\cdot)\approx\max\left\{\min\left\{\int_{0}^L \tilde{K}(\cdot,y)\left(\int_{\R}\widetilde{f}_{n-1}(y,v)\ud v-1\right)\ud y,2L\right\},-2L\right\},
\end{equation}
 provided that $\widetilde{f}_{n-1}$ is a numerical solution of $f(t_{n-1})$. 
 The notation $``\approx"$ in \eqref{eq:Electric} indicates that
 the integrals on the right hand could be computed via the numerical integration based on a mesh for the phase space $\mathbb{T}\times\R$. 
Therefore,  a full discrezation for \eqref{eq:Vla}-\eqref{eq:VPE} is necessary  to sequentially solve $\{\widetilde E_{n}\}_{n\le N}$ and $\{\widetilde{f}_{n}\}_{n\le N}$.

\subsection{Lagrange first-order interpolation}\label{S3.2}
Since the initial density $f_0$ satisfies the normalized condition $\|f_0\|_{L^1_{x,v}}=1$, one can expect that $f_0$
decays to $0$ at infinity with respect to the velocity variable. 
Thus, we make the following hypothesis, which is equivalent to that
$\sup_{x\in\mathbb{T}^d}|f_0(x,v)|\to 0$ as $\|v\|\to\infty$.
\begin{hypothesis}\label{Hyp:initial}
For any $0<\epsilon_0\ll 1$, there exists a positive constant $U_0:=U_0(\epsilon_0)$ such that 
$|f_0(x,v)|<\epsilon_0$ for any $x\in\mathbb T^d$ and 
$v\notin[-U_0,U_0]^d.$
\end{hypothesis}

Let $\delta y=(\delta y_1,\ldots,\delta y_{2d}):=(\delta x_1,\ldots,\delta x_d,\delta v_1,\ldots,\delta v_d)$
 with
$\delta x_i$ and $\delta v_i$ being the stepsizes for the $i$th components of position and velocity coordinates, respectively. 
Under Hypothesis \ref{Hyp:initial}, for any fixed $t\in[0,T]$, it is possible to select a bounded set $\mathscr{D}_b:=\mathbb{T}^d\times[-b,b]^d$ with some $b\in(0,\infty)$  such that the state of particles is mainly concentrated on $\mathscr{D}_b$. Throughout this paper, we assume that $L/\delta x_i\in\mathbb{Z}$, $U_0/\delta v_i\in\mathbb{Z}$, and $b/\delta v_i\in\mathbb{Z}$ for each $i=1,\ldots,d$.
To encompass the phase space $\mathscr{X}:=\mathbb{T}^d\times\R^d$, we adopt the convention $\mathscr{D}_\infty:=\mathscr{X}$.
In this subsection, we introduce the Lagrange first-order interpolation on the domain $\mathscr{D}_b$ with $b\in(0,\infty]$, which can be used as a reconstruction procedure of the semi-Lagrangian method for \eqref{eq:Vla} later.
 For general Lagrange interpolation and spline interpolation, we refer to, e.g., \cite[section 5]{BM08}.

The set of grid points of a rectangular partition for $\mathscr{D}_b$  is given by
\begin{align}\label{eq:MDb}
\mathcal M_{\mathscr D_b}&:=\{(j_1\delta x_1,\ldots, j_d\delta x_d,k_1\delta v_1,\ldots,k_d\delta v_d)\mid \\\notag
&\qquad\qquad\qquad j_i,k_i\in\mathbb{Z} \text{ with } 0\le j_{i}\le L/\delta x_i,|k_i|\le b/\delta v_i,i=1,\ldots,d\}.
\end{align} 
For $\boldsymbol\xi=(\xi_1,\ldots,\xi_{2d})\in\mathbb{Z}^{2d}$, we denote $\square_{\boldsymbol\xi}:=\prod_{i=1}^{2d}[\xi_i\delta y_i,(\xi_i+1)\delta y_i]$. For a function $g$ defined on vertices of the rectangular $\square_{\boldsymbol{\xi}}$,
we define
\begin{equation}\label{eq:Rycell}
\mathscr{R}_{\delta y}^{\boldsymbol\xi}g(y):=\sum_{l_1=\xi_1}^{\xi_1+1}\cdots\sum_{l_{2d}=\xi_{2d}}^{\xi_{2d}+1}g(l_1\delta y_1,\ldots,l_{2d}\delta y_{2d})\prod_{i=1}^{2d}\ell^{(\xi_i)}_{l_i,\delta y_i}(y_i)
\end{equation}
for $y=(y_1,\ldots,y_{2d})\in \square_{\boldsymbol\xi}$, where 
 the associated Lagrange basis satisfies 
$$\ell^{(\xi_i)}_{\xi_i,\delta y_i}(z)=\frac{(\xi_i+1)\delta y_i-z}{\delta y_i},\qquad \ell^{(\xi_i)}_{\xi_i+1,\delta y_i}(z)=\frac{z-\xi_i\delta y_i}{\delta y_i},\quad z\in[\xi_i\delta y_i,(\xi_i+1)\delta y_i].$$
 Given a function $g$ defined on $\mathcal M_{\mathscr D_b}$, we define the Lagrange first-order interpolation operator $\mathscr{R}_{\mathscr{D}_b}^{\delta y}$ applied to $g$ via
\begin{equation}\label{eq:blinearinterp}
\mathscr{R}_{\mathscr{D}_b}^{\delta y} g(y)=\sum_{\mathbf{j},\mathbf{k}}\mathscr{R}_{\delta y}^{(\mathbf{j},\mathbf{k})}g(y) \chi_{\mathbf{j},\mathbf{k}}(y),\quad y\in\mathscr{D}_b,
\end{equation}
where $\mathscr{R}_{\delta y}^{(\mathbf{j},\mathbf{k})}$ is defined in \eqref{eq:Rycell} with $\boldsymbol\xi=(\mathbf{j},\mathbf{k})\in\mathbb{Z}^{2d}$. 
The sum in \eqref{eq:blinearinterp} is taken over the set 
\begin{align*}
\Big\{\mathbf{j}=(j_1,\ldots,j_d),\mathbf{k}=(k_1,\ldots,k_d)&\mid j_i=0,1,\ldots,\frac{L}{\delta x_i}-1,\\
&k_{i}=-\frac{b}{\delta v_i},-\frac{b}{\delta v_i}+1,\ldots,\frac{b}{\delta v_i}-1~\text{for~}i=1,\ldots,d \Big\},
\end{align*}
and $\chi_{\mathbf{j},\mathbf{k}}$ is the indicator function on the rectangular $
\square_{(\mathbf{j},\mathbf{k})}$. It is known that the interpolation error of the Lagrange first-order interpolation satisfies that
 (see, e.g., \cite[Theorem 16.1]{CP91})
\begin{equation}\label{eq:Inerror}\|g-\mathscr{R}_{\mathscr{D}_b}^{\delta y} g\|_{L^\infty(\mathscr{D}_b)}\le Ch^2\|D^2g\|_{L^\infty(\mathscr{D}_b)}
\end{equation}
where $h:=|\delta y|_{\ell^\infty}$ with $|\cdot|_{\ell^\infty}$ being the $\infty$-norm for vectors.

\subsection{Dynamic domain semi-Lagrangian method}\label{S3.3}

In the sequel,
we assume that the numerical integrator for \eqref{eq:SDE} satisfies that
\begin{equation}\label{eq:Xin}
|\widehat{\Psi}^2_{t_{n-1},t_n}(x,v)-v|_{\ell^\infty}\le \tau \|E\|_{L^\infty([0,T]\times\mathbb{T}^d;\ell^\infty)}+\sum_{k=1}^{K}\|\sigma_k\|_{L^\infty(\mathbb{T}^d;\ell^\infty)}|\delta\beta_{n,k}|=:\widetilde\Xi_n.
\end{equation}
This condition is fulfilled by \eqref{eq:inv-Phi}, \eqref{eq:inv-LTS}, and \eqref{eq:inv-SS}. The assumption \eqref{eq:Xin} ensures that in the $n$th time step, the movement of the numerical solution of  \eqref{eq:SDE} in each velocity direction does not exceed $\widetilde\Xi_n$. 
Thanks to \eqref{eq:semi-dis} and \eqref{eq:Xin}, if $\textup{supp}(f_{n-1}^\tau)\subset\mathbb{T}^d\times[-C_1,C_1]^d$ for some $C_1>0$, then
$\textup{supp}(f^\tau_{n})\subset\mathbb{T}^d\times[-C_1-\widetilde\Xi_n,C_1+\widetilde\Xi_n]^d$.
By iterations, if $\textup{supp}(f_0^\tau)\subset\mathbb{T}^d\times[-U_0,U_0]^d$, then $\textup{supp}(f^\tau_{N})\subset\mathbb{T}^d\times\left[-U_0-\Xi,U_0+\Xi\right]^d$
with $\Xi:= \sum_{n=1}^{N}\widetilde\Xi_n$.

 \begin{remark}\label{rem:1}
 Notice that in the deterministic case (i.e., $\sigma\equiv 0$), $\Xi$ could be bounded by a constant depending on $U_0$, $T$ and $E$.
However, in the stochastic case, the particle's trajectories are not uniformly bounded across all samples since $\{\delta\beta_{n,k}\}_{k=1}^{K}$ is not uniformly bounded. To see this, one can use 
the truncation of the increment of the Brownian motion (i.e., $\widehat{\delta \beta_{n,k}}:=\max\{\min\{\delta \beta_{n,k},A_\tau\},-A_\tau\}$)
 instead of $\delta \beta_{n,k}$ in the stochastic characteristics solver, where $A_\tau=\mathcal{O}(\sqrt{\tau|\log(\tau)|})$ (see, e.g., \cite{MT21} for more details). By this substitution, one can apply the traditional semi-Lagrangian method to solve \eqref{eq:Vla} on the domain $\mathbb{T}^d\times[-U_0-nA_\tau,U_0+nA_\tau]^d$ at $n$th time step. However, its computational cost is very expensive, especially when the time stepsize $\tau$ is small and the terminal time $T$ is large (see Tab.\ \ref{Tabletime} in section \ref{S:NE} for a comparison).
\end{remark}

To reduce the computational cost, inspired by \cite{LR23} we propose a dynamic domain adaptation strategy to update the computational domain in phase space adaptively at each time step. We  incorporate 
the region in phase space that includes the positions and velocities for most particles into the computational domain $\mathbb{T}^d\times [-U_n,U_n]^d$ in phase space at $n$th step. Here, $U_0$ is chosen according to Hypotheis \ref{Hyp:initial}.

\begin{algorithm}[htb]
\caption{Dynamic domain semi-Lagranigan method for \eqref{eq:Vla}.}\label{Algo:2}
\begin{itemize}[leftmargin=*]
\item[]
Set up the threshold $\epsilon_0>0$, time stepsize $\tau$,  discretization parameter  $\delta y=(\delta x_1,\ldots,\delta x_d,\delta v_1,\ldots,\delta v_d)$ in phase space, and the
initial density $f^{\delta y,\tau,\epsilon_0}_{0}:=f_0$ in the initial computational domain $\mathbb{T}^d\times[-U_0,U_0]^d$.

\item[] 
\textbf{for} $n=1, \cdots,N$, \textbf{do}

\begin{enumerate}
\item[(1)]
For Case (II), numerically solve $E(t_{n-1},\cdot)$ by \eqref{eq:VPE} and use $f(t_{n-1})$ as an approximation of $\widetilde{f}_{n-1}$ (Skip this step for Case (I));

\item[(2)]
Compute $U_n$ defined by \eqref{eq:Un} and set the phase space mesh
$\mathcal M_{\mathscr{X}_n}$.

\item[(3)]
For each grid point $(x^*,v^*)\in\mathcal M_{\mathscr{X}_n}$, solve the inverse numerical flow of \eqref{eq:SDE} $$(X,V)=\widehat{\Psi}_{t_{n-1},t_{n}}(x^*,v^*)$$ 
via a volume-preserving integrator (\eqref{eq:inv-Phi}, \eqref{eq:inv-LTS} or \eqref{eq:inv-SS}).  Then compute
\begin{equation*}%\label{eq:fStauep0}
 \tilde{f}^{\delta y,\tau,\epsilon_0}_{n}(x^*,v^*):= 
 \begin{cases}
 f^{\delta y,\tau,\epsilon_0}_{n-1}(X,V),\quad
 &\text{if } \quad |V|_{\ell^\infty}\le U_{n-1},\\
 0,\quad &\text{otherwise};
 \end{cases}
\end{equation*}

\item[(4)]
Perform an interpolation (either the Lagrangian interpolation or spline interpolation)
 $f^{\delta y,\tau,\epsilon_0}_{n}=
\mathscr{R}_{\mathscr{X}_n}^{\delta y}\tilde{f}^{\delta y,\tau,\epsilon_0}_{n}$ on $\mathscr{X}_n$.
\end{enumerate}

\textbf{end for}

\item[] 
$\{f_n^{\delta y,\tau,\epsilon_0}\}_{n=0,1,\ldots,N}$ is the computed solution for \eqref{eq:Vla} at time grids.
\end{itemize}

\end{algorithm}
In the rest of this paper, we assume without loss of generality that $L/\delta x_i\in\mathbb{N}$ and $U_0/\delta v_i\in\mathbb{N}$ for all $i=1,\ldots,d$. For the sake of notation simplicity, we denote by $\mathscr{X}_n:=\mathbb{T}^d\times[-U_n,U_n]^d$ and $\mathcal{M}_{\mathscr{X}_n}$ the set of grid points of a rectangular
 partition for $\mathscr{X}_n$ (see \eqref{eq:MDb}).
Recursively, we define 
\begin{equation}\label{eq:Un}
 U_{n}:=\begin{cases}
 U_{n-1}, \quad & \text{if } |f^{\delta y,\tau,\epsilon_0}_{n-1}(x,v)|\le \epsilon_0 
 \quad \forall~(x,v)\in \hat{\Omega} \cap \mathcal{M}_{\mathscr X_{n-1}}, \\
 U_{n-1}+\Xi_{n}, \quad &\text{otherwise},
 \end{cases}
\end{equation}
where 
$\hat{\Omega} := \mathbb{T}^d\times([-U_{n-1},-\hat{U}_{n-1}]\cup[\hat{U}_{n-1},U_{n-1}])^d$ and $\hat{U}_{n-1}:=U_{n-1}-\Xi_{n}$. Here, 
$\Xi_{n}:\Omega\to\R$ could be chosen as any random variable 
satisfying $\Xi_{n}\ge\widetilde\Xi_n$ and $\Xi_{n}/\delta v_i\in\mathbb{N}$ for all $i=1,\ldots,d$, where
$\widetilde\Xi_n$ is defined in \eqref{eq:Xin}.
Finally, we summarize the above ingredients to propose the dynamic domain semi-Lagrangian method in 
Algorithm \ref{Algo:2}.

\section{Convergence analysis}\label{S4}

In this section, we present the convergence analysis of the dynamic domain semi-Lagrangian method (i.e., Algorithm \ref{Algo:2}) for the stochastic linear Vlasov equation (i.e., Case (I)). We use $C$ to denote a generic positive constant that may change from one place to another and depend on several parameters but never on $\tau$, $\delta y$, and $\epsilon_0$.

For $0\le r<t\le T$, let $\mathscr{F}_{r,t}$ be the $\sigma$-algebra generated by $\{\beta(s_1)-\beta(s_2),r\le s_1\le s_2\le t\}$.
For any non-random $(x,v)\in\mathbb{T}^d\times \R^{d}$,
$\widehat{\phi}_{t,T}(x,v)\in\mathscr{F}_{t,T}$ for any $0\le t\le T$.
If $E$ and $\sigma$ are $i$th continuously differentiable with bounded derivatives ($i\in\mathbb{N}_+$), then for any $p\ge2$, there exists $C:=C(p,T)$ such that for any $0\le t\le T$ and $(x,v)\in\mathbb{T}^d\times \R^{d}$,
\begin{align}\label{eq:phideri}
&\E\left[\|D^i\phi_{0,t}(x,v)\|^p\right]+\E\left[\|D^i\widehat{\phi}_{0,t}(x,v)\|^p\right]\le C.
\end{align}
For $i=1,2$, $D\phi_{0,t}(x,v)$ and $D^2\phi_{0,t}(x,v)$ denote the Jacobian matrix and Hessian tensor of $\phi_{0,t}(x,v)$, respectively.
We refer to \cite[Theorem 3.4.4]{KH19} for the proof of \eqref{eq:phideri}.
For any non-random $(x,v)\in\mathbb{T}^d\times \R^{d}$, it can be proved by induction that
$\widehat{\Psi}_{t_n,T}(x,v)\in\mathscr{F}_{t_n,T}$ for any $n=0,1,\ldots,N$.

\begin{hypothesis}\label{Asp:2}
Assume that the inverse numerical flow $\{\Psi_{0,t_n}\}_{n=0}^N$ for \eqref{eq:SDE} satisfies the following conditions. 
\begin{enumerate}
\item[(1)] (Stability) For $j=1,2$ and any $p\ge2$,
\begin{align*}
&\E\left[\|D^j\Psi_{0,t_n}(x,v)\|^p\right]+\E\left[\|D^j\widehat{\Psi}_{0,t_n}(x,v)\|^p\right]\le C(p),\quad (x,v)\in\mathbb{T}^d\times \R^{d}.
\end{align*}
\item[(2)] (Convergence)
There exists  constants $Q>0$, $\gamma\ge 1$, and $p_0\ge 1$ such that for any $(x,v)\in\mathbb{T}^d\times \R^{d}$,
\begin{equation*}
\left(\E\left[\|\widehat{\Psi}_{0,t_n}(x,v)-\widehat{\phi}_{0,t_n}(x,v)\|^{2p_0}\right]\right)^{\frac{1}{2p_0}}\le C(p_0)\tau^{Q}(1+\|x\|^{2\gamma}+\|v\|^{2\gamma})^{\frac12}.
\end{equation*}
\end{enumerate}
\end{hypothesis}

By the Lipschitz continuity of $f_0$, 
Hypothesis \ref{Asp:2} implies that the discretization error of \eqref{eq:semi-dis} has strong convergence order $Q$ in the sense that for any $(x,v)\in\mathbb{T}^d\times \R^{d}$,
\begin{equation}\label{eq:SC-semi}
\sup_{1\le n\le N}\left(\E\left[|f_n^\tau(x,v)-f(t_n,x,v)|^{2p_0}\right]\right)^{\frac{1}{2p_0}}\le C\tau^{Q}(1+\|x\|^{2\gamma}+\|v\|^{2\gamma})^{\frac12}.
\end{equation}
The following lemma implies that 
\eqref{eq:semi-dis} with 
\eqref{eq:inv-Phi}, 
 \eqref{eq:inv-LTS} or \eqref{eq:inv-SS},
has first-order convergence of accuracy.
\begin{lemma}\label{lem:local} 
Assume that the functions $E$ and $\sigma$ are twice continuously differentiable with bounded derivatives.
Then \eqref{eq:inv-Phi}, \eqref{eq:inv-LTS}, and \eqref{eq:inv-SS} satisfy Hypothesis \ref{Asp:2} with $Q=\gamma=1$ and all $p_0\ge1$.
\end{lemma}
\begin{proof}
We only give a sketch of the proof for the convergence condition in Hypothesis \ref{Asp:2}(2) since the proof of the stability condition in Hypothesis \ref{Asp:2}(1) is standard (see, e.g., \cite[Lemma 1.5]{MT21}).
For any $0\le s< t\le T$ and $(x,v)\in\mathbb{T}^d\times \R^{d}$,
$$\widehat{\phi}_{s,t}(x,v)=\left(\begin{array}{c}x \\v\end{array}\right)+\int_{s}^{t} \left(\begin{array}{c}\widehat{\phi}_{r,t}^2(x,v) \\E(r,\widehat{\phi}_{r,t}^1(x,v))\end{array}\right)\ud r+\int_{s}^{t} \left(\begin{array}{c}0 \\\sigma(\widehat{\phi}_{r,t}^1(x,v))\end{array}\right)\odot\ud \beta(r),$$
where $\widehat{\phi}_{r,t}^1(x,v)$ and $\widehat{\phi}_{r,t}^2(x,v)$ denote the components of $\widehat{\phi}_{r,t}(x,v)$ in the physical space and velocity space, respectively.
Here, the stochastic integral is understood in the sense of the backward Stratonovitch integral. It coincides with the backward It\^o integral since the quadratic covariation between $\widehat{\phi}_{r,t}^1(x,v)$ and $\{\beta_k\}_{k=1}^K$ vanishes (see \cite{PP87} for more details).
According to the fundamental mean-square convergence theorem (see, e.g., \cite[Theorem 1.1.1]{MT21} and \cite[Theorem 2.1]{TZ13}), 
to prove Hypothesis \ref{Asp:2}(2) with $Q=\gamma=1$ and $p_0\ge1$,
it suffices to show that  for any $\tau\le t\le T$ and $(x,v)\in\mathbb{T}^d\times \R^{d}$,
\begin{gather}\label{eq:localmean}
\left\|\E\left[\widehat{\Psi}_{t-\tau,t}(x,v)-\widehat{\phi}_{t-\tau,t}(x,v)\right]\right\|\le C\tau^2(1+\|x\|^2+\|v\|^2)^{\frac12},\\\label{eq:localmq}
\left(\E\left[\|\widehat{\Psi}_{t-\tau,t}(x,v)-\widehat{\phi}_{t-\tau,t}(x,v)\|^{2p_0}\right]\right)^{\frac{1}{2p_0}}\le C(p_0)\tau^{\frac32}(1+\|x\|^2+\|v\|^{2})^{\frac{1}{2}}.
\end{gather}
By the property that the expectation of the backward It\^o integral vanishes, it can be shown that the one-step maps generated by \eqref{eq:inv-Phi}, \eqref{eq:inv-LTS}, and \eqref{eq:inv-SS}
satisfy \eqref{eq:localmean} and \eqref{eq:localmq}. We omit further details since the procedure is the same as in the forward flow setting 
(see, e.g., \cite[Theorem 4]{CCDL20}).
\end{proof}

In the following, let $f_n^{\delta y,\tau,\epsilon_0}$ be generated by Algorithm \ref{Algo:2} with the Lagrange first-order interpolation.
To study the mean-square error of Algorithm \ref{Algo:2}, we introduce an auxiliary sequence $\{F_n^{\delta y,\tau,\epsilon_0}\}_{n=0}^N$.
Let $\mathcal {M}_{\mathscr{X}}$ stand for the set of grid points of a rectangular partition for $\mathscr{X}=\mathbb{T}^d\times\R^d$ (see \eqref{eq:MDb}). 
Define 
$\tilde{F}^{\delta y,\tau,\epsilon_0}_{n}(x^*,v^*):=0$ for $(x^*,v^*)\in\mathcal{M}_\mathscr{X}\cap\mathscr{X}_n^{c}$ and $\tilde{F}^{\delta y,\tau,\epsilon_0}_{n}(x^*,v^*):={f}^{\delta y,\tau,\epsilon_0}_{n}(x^*,v^*)$ for $(x^*,v^*)\in \mathcal{M}_\mathscr{X}\cap\mathscr{X}_n$. Then for any $n=0,1,\ldots,N$, we define
\begin{align*}%\label{eq:F}
F^{\delta y,\tau,\epsilon_0}_{n}:=(\mathscr{R}_{\mathscr{X}}^{\delta y}\tilde{F}^{\delta y,\tau,\epsilon_0}_{n})\cdot\mathbf{1}_{\mathscr{X}_n}.
\end{align*}
 Note that $F^{\delta y,\tau,\epsilon_0}_{n}$
is the zero extension of $f^{\delta y,\tau,\epsilon_0}_{n}$ to $\mathscr{X}$.
 Specially, $F^{\delta y,\tau,\epsilon_0}_{n}=f^{\delta y,\tau,\epsilon_0}_{n}$ on $\mathscr{X}_n$ for every $n=0,1,\ldots,N$.
For any $(x^*,v^*)\in\mathcal{M}_{\mathscr{X}_n}$,
$f^{\delta y,\tau,\epsilon_0}_{n}(x^*,v^*)=F^{\delta y,\tau,\epsilon_0}_{n-1}(\widehat{\Psi}_{t_{n-1},t_{n}}(x^*,v^*)),$
which implies that for any $(x^*,v^*)\in\mathcal{M}_{\mathscr{X}}$,
\begin{equation*}
\tilde{F}^{\delta y,\tau,\epsilon_0}_n(x^*,v^*)=F^{\delta y,\tau,\epsilon_0}_{n-1}(\widehat{\Psi}_{t_{n-1},t_{n}}(x^*,v^*))\cdot\mathbf{1}_{\mathscr{X}_n}(x^*,v^*).
\end{equation*}
Hence, Algorithm \ref{Algo:2} can be summarised as
\begin{align}\label{eq:AdSL}
F^{\delta y,\tau,\epsilon_0}_{n}=(\mathscr{R}_{\mathscr{X}}^{\delta y}([F^{\delta y,\tau,\epsilon_0}_{n-1}\circ\widehat{\Psi}_{t_{n-1},t_{n}}]\cdot\mathbf{1}_{\mathscr{X}_n}))\cdot\mathbf{1}_{\mathscr{X}_n}
\end{align}
for $n=1,\ldots,N$
with the initial value $F^{\delta y,\tau,\epsilon_0}_{0}=f_0\cdot \mathbf{1}_{\mathscr{X}_0}$. Here, $\delta y,\tau$, and $\epsilon_0$ are given in Algorithm \ref{Algo:2}.

\begin{theorem}\label{thm:DSLC}
Assume that Hypotheses \ref{Asp:1}(A1)-(A2) and \ref{Hyp:initial} hold. 
Let $F^{\delta y,\tau,\epsilon_0}_N$ be given by \eqref{eq:AdSL} with 
the inverse numerical flow $\{\widehat{\Psi}_{t_{n-1},t_n}\}_{n=0}^N$ satisfying Hypothesis \ref{Asp:2}, and let $\mathscr{R}_{\mathscr{X}}^{\delta y}$
be the Lagrange first-order interpolation.
Then there exists a positive constant $C$ such that for any $R>0$,
\begin{equation*}%\label{eq:DSLC}
\max_{(x,v)\in\mathbb{T}^d\times [-R,R]^d}\left(\E\left[|F^{\delta y,\tau,\epsilon_0}_N(x,v)-f(T,x,v)|^2\right]\right)^{\frac12}\le CR^\gamma\tau^Q +C\tau^{-1}(h^2+\epsilon_0),
\end{equation*}
where $\gamma$ and $Q$ are given in Hypothsis \ref{Asp:2}.
\end{theorem}

\begin{proof}
Thanks to \eqref{eq:SC-semi} and the triangle inequality, for any fixed $R>0$,
\begin{align}\label{eq:intermediate}
&\max_{(x,v)\in\mathbb{T}^d\times [-R,R]^d}\|F_N^{\delta y,\tau,\epsilon_0}(x,v)-f(T,x,v)\|_{L^2(\Omega)}\\\notag
&\le \max_{(x,v)\in\mathbb{T}^d\times [-R,R]^d}\|F_N^{\delta y,\tau,\epsilon_0}(x,v)-f_N^\tau(x,v)\|_{L^2(\Omega)}+CR^\gamma\tau^Q.
\end{align}
For $n=0,1,\ldots,N$, denote $\mathcal{E}_{n}:=F^{\delta y,\tau,\epsilon_0}_{n}-f^\tau_n$. 
For every $n=1,\ldots, N$, it follows from \eqref{eq:AdSL} and \eqref{eq:semi-dis} that for any $(x,v)\in\mathscr{X}$,
\begin{align}\label{eq:error-X}
	\mathcal{E}_{n}&=\mathscr{R}_{\mathscr{X}}^{\delta y}([F^{\delta y,\tau,\epsilon_0}_{n-1}\circ\widehat{\Psi}_{t_{n-1},t_{n}}]\cdot\mathbf{1}_{\mathscr{X}_n})\cdot\mathbf{1}_{\mathscr{X}_n}-f_{n-1}^\tau\circ\widehat{\Psi}_{t_{n-1},t_n}\\\notag
	&=\underbrace{\mathscr{R}_{\mathscr{X}}^{\delta y}(\mathcal{E}_{n-1}\circ\widehat{\Psi}_{t_{n-1},t_n})-\mathscr{R}_{\mathscr{X}}^{\delta y}([F^{\delta y,\tau,\epsilon_0}_{n-1}\circ\widehat{\Psi}_{t_{n-1},t_{n}}]\cdot\mathbf{1}_{\mathscr{X}_n})\cdot\mathbf{1}_{\mathscr{X}_n^c}}_{=:\mathcal{I}_1}\\\notag
	&\quad\underbrace{-\mathscr{R}_{\mathscr{X}}^{\delta y}([F^{\delta y,\tau,\epsilon_0}_{n-1}\circ\widehat{\Psi}_{t_{n-1},t_{n}}]\cdot\mathbf{1}_{\mathscr{X}_n^c})}_{=:\mathcal{I}_2}+\underbrace{\mathscr{R}_{\mathscr{X}}^{\delta y}f^\tau_n- f_{n}^\tau}_{=:\mathcal{I}_3}
\end{align}
due to the linearity of $\mathscr{R}_{\mathscr{X}}^{\delta y}$.

Let $(x,v)\in\mathscr{X}$ be arbitrarily fixed. 
Then there exists $\boldsymbol\xi:=(\mathbf{j},\mathbf{k})\in\mathbb{Z}^{2d}$ such that $(x,v)\in\square_{(\mathbf{j},\mathbf{k})}$.
By \eqref{eq:blinearinterp} and \eqref{eq:Rycell},
for $y:=(x,v)\in\square_{(\mathbf{j},\mathbf{k})}$,
\begin{align*}
&\|\mathscr{R}_{\mathscr{X}}^{\delta y}(\mathcal{E}_{n-1}\circ\widehat{\Psi}_{t_{n-1},t_n})(x,v)\|_{L^2(\Omega)}\\
&\le\sum_{l_1=\xi_1}^{\xi_1+1}\cdots\sum_{l_{2d}=\xi_{2d}}^{\xi_{2d}+1}\|\mathcal{E}_{n-1}(\widehat{\Psi}_{t_{n-1},t_n}(l_1\delta y_1,\ldots,l_{2d}\delta y_{2d}))\|_{L^2(\Omega)}\prod_{i=1}^{2d}\ell^{(\xi_i)}_{l_i,\delta y_i}(y_i).
\end{align*}
Since $\widehat{\Psi}_{t_{n-1},t_n}\in\mathscr{F}_{t_{n-1},t_n}$ and $\mathcal{E}_{n-1}\in\mathscr{F}_{0,t_{n-1}}$ are independent, we obtain that for any $(l_1\delta y_1,\ldots,l_{2d}\delta y_{2d})\in\mathcal M_{\mathscr{X}}$,
\begin{align*}
&\E\left[|\mathcal{E}_{n-1}(\widehat{\Psi}_{t_{n-1},t_n}(l_1\delta y_1,\ldots,l_{2d}\delta y_{2d}))|^2\right]\\
&=\E\left[\E\left[|\mathcal{E}_{n-1}(X,V)|^2\right]\mid_{(X,V)=\widehat{\Psi}_{t_{n-1},t_n}(l_1\delta y_1,\ldots,l_{2d}\delta y_{2d})} \right]\\
& \le
\max_{(\bar{x},\bar{v})\in\mathscr{X}}\E\left[|\mathcal{E}_{n-1}(\bar{x},\bar{v})|^2\right].
\end{align*}
By
$\ell^{(\xi_i)}_{\xi_i,\delta y_i}(z)+\ell^{(\xi_i)}_{\xi_i+1,\delta y_i}(z)=1$ for any $z\in[\xi_i\delta y_i,(\xi_i+1)\delta y_i]$, it then holds that 
\begin{equation}\label{eq:En-1}
\|\mathscr{R}_{\mathscr{X}}^{\delta y}(\mathcal{E}_{n-1}\circ\widehat{\Psi}_{t_{n-1},t_n})(x,v)\|_{L^2(\Omega)}\le\max_{(\bar{x},\bar{v})\in\mathscr{X}}\|\mathcal{E}_{n-1}(\bar{x},\bar{v})\|_{L^2(\Omega)}.
\end{equation}
Next, we estimate the error terms in \eqref{eq:error-X}, where $\mathcal{I}_1$ and $\mathcal{I}_2$ correspond to the truncation error, and  $\mathcal{I}_3$ is related to the interpolation error.

\textit{Interpolation error.}
We begin with a pointwise error estimate for the interpolation operator defined in \eqref{eq:Rycell}.
Let $\boldsymbol\xi=(\xi_1,\ldots,\xi_{2d})\in\mathbb{Z}^{2d}$ with $\xi_i\in\{0,1,\ldots,\frac{L}{\delta x_i}-1\}$ for $i=1,\ldots,d$. 
For each $r=0,1,\ldots,2d$, we denote 
\begin{equation*}
\mathscr{R}_{\delta y}^{(\xi_1,\ldots,\xi_r)}g(y):=\sum_{l_1=\xi_1}^{\xi_1+1}\cdots\sum_{l_r=\xi_r}^{\xi_r+1} g(l_1\delta y_1,\ldots,l_r\delta y_r,y_{r+1},\ldots,y_{2d}
)\prod_{i=1}^r\ell^{(\xi_i)}_{l_i,\delta y_i}(y_i) 
\end{equation*}
for $y\in \mathscr{X}$ with $(y_1,\ldots,y_r)\in\prod_{i=1}^{r}[\xi_i\delta y_i,(\xi_i+1)\delta y_i]$ so that 
\begin{align*}
g(z)-\mathscr{R}_{\delta y}^{\boldsymbol\xi}g(z)
=\sum_{r=0}^{2d-1}\left(\mathscr{R}_{\delta y}^{(\xi_1,\ldots,\xi_r)}g(z)-\mathscr{R}_{\delta y}^{(\xi_1,\ldots,\xi_{r+1})}g(z)\right),\quad z\in\square_{\boldsymbol\xi}.
\end{align*}
If $(y_1,\ldots,y_{r+1})\in \prod_{i=1}^{r+1}[\xi_i\delta y_i,(\xi_i+1)\delta y_i]$, then
\begin{align*}
&\mathscr{R}_{\delta y}^{(\xi_1,\ldots,\xi_r)}g(y)-\mathscr{R}_{\delta y}^{(\xi_1,\ldots,\xi_{r+1})}g(y)=\sum_{l_1=\xi_1}^{\xi_1+1}\cdots\sum_{l_r=\xi_r}^{\xi_r+1}
\delta _{\xi_1,\ldots,\xi_r}^{l_1,\ldots,l_{r}} g(y)
\prod_{i=1}^r\ell^{(\xi_i)}_{l_i,\delta y_i}(y_i),
\end{align*}
where 
\begin{align*}
\delta _{\xi_1,\ldots,\xi_{r}}^{l_1,\ldots,l_{r}}g(y)&:=g(l_1\delta y_1,\ldots,l_r\delta y_r,y_{r+1},\ldots,y_{2d})\\
&\quad-\sum_{l_{r+1}=\xi_{r+1}}^{\xi_{r+1}+1}g(l_1\delta y_1,\ldots,l_{r+1}\delta y_{r+1},y_{r+2},\ldots,y_{2d})\ell^{(\xi_{r+1})}_{l_{r+1},\delta y_{r+1}}(y_{r+1}).
\end{align*}
By the mean value theorem, for any twice continuously differentiable function $G:[s_1,s_2]\to \R$,  $G(s)-G(s_1)\frac{s-s_{2}}{s_1-s_2}-G(s_2)\frac{s-s_{1}}{s_2-s_1}$ equals
$$
\int_0^1\int_0^{\theta_1}G^{\prime\prime}(s_1+\theta_1(s_2-s_1)+\theta_2(s-s_2))\ud \theta_2\ud \theta_1(s-s_2)(s-s_1)$$
for $s\in[s_1,s_2]$. Consequently,
for any $y=(y_1,\ldots,y_{2d})\in\square_{\boldsymbol\xi}$, the interpolation error is bounded as
\begin{align}\label{eq:RgInte}
&|g(y)-\mathscr{R}_{\mathscr{X}}^{\delta y}g(y)|=|g(y)-\mathscr{R}_{\delta y}^{\boldsymbol\xi}g(y)|\\\notag
&\le \sum_{r=0}^{2d-1}\sum_{l_1=\xi_1}^{\xi_1+1}\cdots\sum_{l_r=\xi_r}^{\xi_r+1}
\int_0^1\int_0^{\theta_1}|\partial_{y_{r+1}}^2g(l_1\delta y_1,\ldots,l_r\delta y_r,y^{l_{r+1},\theta_1,\theta_2}_{r+1},\\\notag
&\qquad\qquad\qquad\qquad\qquad\qquad\qquad y_{r+2},\ldots,y_{2d})|\ud\theta_1\ud\theta_2|\delta y_{r+1}|^2
\prod_{i=1}^r\ell^{(\xi_i)}_{l_i,\delta y_i}(y_i)
\end{align}
with
$y^{l_{r+1},\theta_1,\theta_2}_{r+1}:=l_{r+1}\delta y_{r+1}+
\theta_1\delta y_{r+1}+\theta_2(y_{r+1}-(l_{r+1}+1)\delta y_{r+1})$. By the higher order
chain rules of Faà di Bruno type \cite{MR00}, $f_n^\tau(x,v)=f_0(\widehat{\Psi}_{0,t_n}(x,v))$, Hypothesis \ref{Asp:2}(1), and the H\"older inequality, we have
\begin{gather*}
\sup_{(x,v)\in\mathbb{T}^d\times\mathbb{R}^d}\E\left[\|D^2 f_n^\tau(x,v)\|^2\right]\le C\quad\forall~n\in\{0,1,\ldots,N\}.
\end{gather*}
Taking the $\|\cdot\|_{L^2(\Omega)}$-norm on both sides of \eqref{eq:RgInte} with $g=f_n^\tau$ yields
\begin{equation}\label{eq:IntererrorX}
\|\mathscr{R}_{\mathscr{X}}^{\delta y}f^\tau_n(x,v)- f_{n}^\tau(x,v)\|_{L^2(\Omega)}
\le Ch^2,
\end{equation}

\textit{Truncation error.}
The choice of $U_n$ in \eqref{eq:Un} ensures that
\begin{equation}\label{eq:flowerbound}
 |F_{n-1}^{\delta y,\tau,\epsilon_0}(\widehat{\Psi}_{t_{n-1},t_{n}}(x,v))|\le\epsilon_0 \quad 
 \forall~(x,v)\in\mathscr{X}\text{ with } |v|_{\ell^\infty}\ge U_{n}.
\end{equation}
Indeed, $|f_{n-1}^{\delta y,\tau,\epsilon_0}(x,v)|\le \epsilon_0$ for all $(x,v)\in\mathcal{M}_{\mathscr{X}_{n-1}}$ with $|v|_{\ell^\infty}\ge U_{n}-\Xi_n$, and thus 
$|F_{n-1}^{\delta y,\tau,\epsilon_0}(x,v)|\le \epsilon_0$
for all $(x,v)\in\mathscr{X}$ with $|v|_{\ell^\infty}\ge U_{n}-\Xi_n$.
Then by \eqref{eq:Xin}, for any $(x,v)\in\mathscr{X}$ with $|v|_{\ell^\infty}\ge U_{n}$, 
$$|\widehat{\Psi}^2_{t_{n-1},t_{n}}(x,v)|_{\ell^\infty}\ge |v|_{\ell^\infty}-|v-\widehat{\Psi}^2_{t_{n-1},t_{n}}(x,v)|_{\ell^\infty}\ge U_n-\Xi_n,$$
which proves the assertion \eqref{eq:flowerbound}. It follows from \eqref{eq:flowerbound} that
\begin{equation}\label{eq:Truncateerror1}
\|\mathscr{R}_{\mathscr{X}}^{\delta y}([F^{\delta y,\tau,\epsilon_0}_{n-1}\circ\widehat{\Psi}_{t_{n-1},t_{n}}]\cdot\mathbf{1}_{\mathscr{X}_n^c})\|_{L_{x,v}^\infty}\le \epsilon_0,\quad \textup{a.s.}
\end{equation}
If $(x,v)\in\mathscr{X}_n^c$, then there exists $(\mathbf{j},\mathbf{k})\in \mathbb{Z}^{2d}$ with
$0\le j_i\le \frac{L}{\delta x_i}-1$ and $k_i\le - \frac{U_n}{\delta v_i}-1$ or $k_i\ge\frac{U_n}{\delta v_i}$ for $i=1,\ldots,d$ 
such that $(x,v)\in\square_{(\mathbf{j},\mathbf{k})}.$ This, together with \eqref{eq:flowerbound}, yields that for any $(x,v)\in\mathscr{X}_n^c$, 
\begin{align}\label{eq:Truncateerror2}
&\quad\ \left|\mathscr{R}_{\mathscr{X}}^{\delta y}([F^{\delta y,\tau,\epsilon_0}_{n-1}\circ\widehat{\Psi}_{t_{n-1},t_{n}}]\cdot\mathbf{1}_{\mathscr{X}_n})(x,v)\right|\\\notag
&\le \max_{\substack{(x,v)\in\mathcal{M}_{\mathscr{X}},\\|v|_{\ell^\infty}\ge U_n}}|F^{\delta y,\tau,\epsilon_0}_{n-1}(\widehat{\Psi}_{t_{n-1},t_{n}}(x,v))|\le \epsilon_0,\quad \textup{a.s.}
\end{align}

Inserting the estimates \eqref{eq:En-1}, \eqref{eq:IntererrorX}, \eqref{eq:Truncateerror1}, and \eqref{eq:Truncateerror2} into \eqref{eq:error-X} leads to
\begin{align*}
\sup_{(x,v)\in\mathscr{X}}\|\mathcal{E}_{n}(x,v)\|_{L^2(\Omega)}\le \sup_{(x,v)\in\mathscr{X}}\|\mathcal{E}_{n-1}(x,v)\|_{L^2(\Omega)}+Ch^2+2\epsilon_0,
\end{align*}
which, along with the discrete Gronwall inequality and \eqref{eq:intermediate}, completes the proof.
\end{proof}
\begin{remark}\label{rem:secconvergence}

(1) If we directly adopt the analysis used in the deterministic case, the error between $F_n^{\delta y,\tau,\epsilon_0}$ and $f(t_n)$ can be divided in the
following way:
\begin{align*}
F_n^{\delta y,\tau,\epsilon_0}-f(t_n)
&=\mathscr{R}_{\mathscr{X}}^{\delta y}([F^{\delta y,\tau,\epsilon_0}_{n-1}\circ\widehat{\Psi}_{t_{n-1},t_{n}}]\cdot\mathbf{1}_{\mathscr{X}_n})\cdot\mathbf{1}_{\mathscr{X}_n}-f(t_{n-1})\circ\widehat{\phi}_{t_{n-1},t_n}\\\notag
	&=\mathscr{R}_{\mathscr{X}}^{\delta y}([F_{n-1}^{\delta y,\tau,\epsilon_0}-f(t_{n-1})]\circ\widehat{\Psi}_{t_{n-1},t_n})\\
	&\quad-\mathscr{R}_{\mathscr{X}}^{\delta y}([F^{\delta y,\tau,\epsilon_0}_{n-1}\circ\widehat{\Psi}_{t_{n-1},t_{n}}]\cdot\mathbf{1}_{\mathscr{X}_n})\cdot\mathbf{1}_{\mathscr{X}_n^c}\\\notag
	&\quad-\mathscr{R}_{\mathscr{X}}^{\delta y}([F^{\delta y,\tau,\epsilon_0}_{n-1}\circ\widehat{\Psi}_{t_{n-1},t_{n}}]\cdot\mathbf{1}_{\mathscr{X}_n^c})\\
	&\quad+\mathscr{R}_{\mathscr{X}}^{\delta y}[f(t_{n-1})\circ\widehat{\Psi}_{t_{n-1},t_n}]-f(t_{n-1})\circ\widehat{\Psi}_{t_{n-1},t_n}\\
	&\quad+f(t_{n-1})\circ\widehat{\Psi}_{t_{n-1},t_n}\!-\!f(t_{n-1})\circ\widehat{\phi}_{t_{n-1},t_n}=:J_1+J_2+J_3+J_4+J_5.	
\end{align*}
The error terms $J_1$, $J_2$, $J_3$, and $J_4$ can be estimated as shown in the proof of Theorem \ref{thm:DSLC}.  However, for the last error term $J_5$, by using \eqref{eq:localmq} with $t=t_n$, one can only obtain a local error estimate of order $3/2$. Then based on the discrete Gronwall inequality, the error arising from the time discretization achieves only a global convergence order of  $1/2$.
Hence, it is essential to introduce the temporal semi-discretization $f_n^\tau$ as an intermediate in analyzing the discretization error of the full discretization, as demonstrated  in \eqref{eq:intermediate}.

(2) In the estimate of the interpolation error, we use the pointwise equality \eqref{eq:RgInte} instead of the classical interpolation error inequality \eqref{eq:Inerror}. Otherwise, we would presumably need the following estimate
\begin{align*}
\E\left[\sup_{(x,v)\in\mathbb{T}^d\times\R^d}\|D^2\widehat{\phi}_{0,t}(x,v)\|^p\right]\le C,
\end{align*}
instead of relying on the weaker version \eqref{eq:phideri}. 
\end{remark}

The following result partially affirms the conjecture proposed in \cite{BC24}, demonstrating the first-order convergence for the Lie–Trotter splitting scheme in the mean-square sense, particularly for transport noise. The error analysis of Algorithm \ref{Algo:2} for the stochastic Vlasov–Poisson equation (Case II) is more involved, and we will investigate it in the future. 
\begin{corollary}\label{coro:concrete}
Assume that the functions $E$ and $\sigma$ are twice continuously differentiable with bounded derivatives, and that
$f_0$ satisfies Hypothesis \ref{Hyp:initial}. 
Let $F^{\delta y,\tau,\epsilon_0}_N$ be given by \eqref{eq:AdSL} with 
the inverse numerical flow $\{\widehat{\Psi}_{t_{n-1},t_n}\}_{n=0}^N$ generated by  \eqref{eq:inv-Phi}, \eqref{eq:inv-LTS} or \eqref{eq:inv-SS}.  Let $\mathscr{R}_{\mathscr{X}}^{\delta y}$
be the Lagrange first-order interpolation.
Then there exists a positive constant $C$ such that for any $R>0$,
\begin{equation*}
\max_{(x,v)\in\mathbb{T}^d\times [-R,R]^d}\left(\E\left[|F^{\delta y,\tau,\epsilon_0}_N(x,v)-f(T,x,v)|^2\right]\right)^{\frac12}\le CR \tau +C\tau^{-1}(h^2+\epsilon_0).
\end{equation*}
\end{corollary}

\begin{proof}
This is a direct consequence of Lemma \ref{lem:local} and Theorem \ref{thm:DSLC}.
\end{proof}

\section{Numerical experiments}\label{S:NE}
In this section, to verify our theorectial findings, we consider the stochastic Vlasov equation with the following two initial densities
\begin{equation}\label{eq:LLDP}
f_0(x,v)=\frac{1}{\sqrt{2\pi}}\exp(-v^2/2)\Big(1+\alpha \cos(\frac{2\pi}{L}x)\Big)
\end{equation}
and
\begin{equation}\label{eq:TSI}
f_0(x,v)=\frac{1}{\sqrt{2\pi}}v^2\exp(-v^2/2)\Big(1+\alpha \cos(\frac{2\pi}{L}x)\Big)
\end{equation}
with the perturbation parameter $\alpha>0$. The stochastic Vlasov equation with the initial density
\eqref{eq:LLDP} corresponds to the linear Landau damping problem, which describes the damped propagation of a small-amplitude plasma wave (cf.\ \cite{CMS10}). 
The initial density \eqref{eq:TSI} is often used in the two stream instability problem, 
 for which a counter-streaming plasma flow exists in velocity space (cf.\ \cite{MCV17}).

\subsection{Accuracy and efficiency tests}\label{S5.1}
In this subsection, we test the convergence rate and computational cost of Algorithm \ref{Algo:2} with the Lagrange first-order interpolation by using the following example.
\begin{example}\label{exp:CA}
Consider the following stochastic linear Vlasov equation
\begin{align*}
	\ud_tf +\left(v\partial_xf+\cos(2\pi x)\partial_vf\right)\ud t
	+\sigma(x)\partial_v f\odot \ud \beta(t)=0,\quad t\in[0,T]
\end{align*}
in $\mathbb T\times\R$, where $f_0$ is given by \eqref{eq:LLDP} or \eqref{eq:TSI} with $\alpha=0.05$ and $L=1$. Here $\sigma:\mathbb{T}\to\R$ is globally Lipschitz continuous and $\{\beta(t)\}_{t\in[0,T]}$ is a standard 1-dimensional Brownian motion.
\end{example}

\begin{table} [!htbp]
\caption{Errors 
of Algorithm \ref{Algo:2} using the Lagrange first-order interpolation for Example \ref{exp:CA} with $\sigma(x)=0.5\sin(2\pi x)$: the initial density \eqref{eq:LLDP} in the linear Landau damping problem \textup{(}up\textup{)}, the initial density \eqref{eq:TSI} in the two stream instability problem \textup{(}down\textup{)}.}\label{T_error_space}
\centering
\resizebox{1\textwidth}{!}{
\begin{tblr}{
vline{1-5} = {},
cell{1}{1} = {r=2}{valign = m},
cell{1}{2} = {c=3}{halign = m},
cell{3-7}{1} = {c},
cell{1-2}{2-4} = {c}}
\hline
\diagboxthree{{Stepsize}}{Error (Order)}{Method}
& Linear Landau damping problem \\ \hline
& SEM & LTSM & SSM \\ \hline
$\tau = \frac{T}{16},\delta x=\delta v=1/64$ 
& $4.396e-2~(*)$ & $ 4.705e-2~(*)$ & $2.930e-2~(*)$ \\ 
\hline
\begin{tabular}[c]{@{}c@{}} 
$\tau = \frac{T}{32},\delta x=\delta v=1/128$ 
\end{tabular} 
& $2.247e-2~(0.97)$ & $2.415e-2~(0.96)$ & $1.475e-2~(0.99) $\\ 
\hline
\begin{tabular}[c]{@{}c@{}} 
$\tau = \frac{T}{64},\delta x=\delta v=1/256$ 
\end{tabular} 
& $1.155e-2~(0.96)$ & $1.081e-2~(1.16)$ & $6.902e-3~(1.10)$ \\ 
\hline
\begin{tabular}[c]{@{}c@{}} 
$\tau = \frac{T}{128},\delta x=\delta v=1/512$ 
\end{tabular} 
& $5.914e-3~(0.97)$ & $6.009e-3~(0.85)$ & $3.789e-3~(0.87)$ \\ 
\hline
\begin{tabular}[c]{@{}c@{}} 
$\tau = \frac{T}{256},\delta x=\delta v=1/1024$ 
\end{tabular} 
&$3.079e-3~(0.94)$ & $3.042e-3~(0.98)$ & $1.793e-3~(1.08)$ \\
\hline
\end{tblr}
}
\centering
\resizebox{1\textwidth}{!}{
\begin{tblr}{
vline{1-5} = {},
cell{1}{1} = {r=2}{valign = m},
cell{1}{2} = {c=3}{halign = m},
cell{3-7}{1} = {c},
cell{1-2}{2-4} = {c}}
\hline
\diagboxthree{{Stepsize}}{Error (Order)}{Method}
& Two stream instability problem \\ \hline
& SEM & LTSM & SSM \\ \hline
$\tau = \frac{T}{16},\delta x=\delta v=1/64$ 
& $5.155e-2~(*)$ & $ 5.020e-2~(*)$ & $3.576e-2~(*)$ \\ 
\hline
\begin{tabular}[c]{@{}c@{}} 
$\tau = \frac{T}{32},\delta x=\delta v=1/128$ 
\end{tabular} 
& $2.950e-2~(0.80)$ & $2.995e-2~(0.75)$ & $1.977e-2~(0.86) $\\ 
\hline
\begin{tabular}[c]{@{}c@{}} 
$\tau = \frac{T}{64},\delta x=\delta v=1/256$ 
\end{tabular} 
& $1.603e-2~(0.88)$ & $1.382e-2~(1.12)$ & $9.443e-3~(1.07)$ \\ 
\hline
\begin{tabular}[c]{@{}c@{}} 
$\tau = \frac{T}{128},\delta x=\delta v=1/512$ 
\end{tabular} 
& $7.830e-3~(1.03)$ & $7.026e-3~(0.98)$ & $5.088e-3~(0.89)$ \\ 
\hline
\begin{tabular}[c]{@{}c@{}} 
$\tau = \frac{T}{256},\delta x=\delta v=1/1024$ 
\end{tabular} 
&$4.114e-3~(0.93)$ & $3.467e-3~(1.02)$ & $2.459e-3~(1.05)$ \\
\hline
\end{tblr}
}
\end{table}

To validate the error bound in Theorem \ref{thm:DSLC} with $Q=1$, we maintain the ratio $h/\tau$ as a constant to test the discretization error of Algorithm \ref{Algo:2} for Example \ref{exp:CA} with $\sigma(x)=0.5\sin(2\pi x)$. In Tab.\ \ref{T_error_space}, we present the computational result of the error
\begin{align}\label{eq:Ehtau}
 \sup_{(x,v)\in[0,1]\times[-1,1]}\left(\E\left[|f_N^{\delta y,\tau,\epsilon_0}(x,v)-f_N^{2\delta y,2\tau,\epsilon_0}(x,v)|^2\right]\right)^{\frac12}
\end{align}
with $\delta y=(\delta x,\delta v)$ at $T = 1$ for a series of different stepsizes, where $\epsilon_0=f_0(0,U_0)$ with $U_0=6$, the supremum is obtained at the coarsest grid points in phase space, and the expectation is realized by 200 samples.
 The result in Tab.\ \ref{T_error_space} coincides with the error bound $C\tau\left[(h/\tau)^2+1\right]$ in Theorem \ref{thm:DSLC} with $Q=1$ (see also Corollary \ref{coro:concrete}). 
 As shown in Tab.\ \ref{T_error_space},
the discretization error of Algorithm \ref{Algo:2} with \eqref{eq:inv-SS} is smaller than those with
\eqref{eq:inv-Phi} and \eqref{eq:inv-LTS}. 

\begin{table}
 \caption{CPU time\textup{(}s\textup{)} of Algorithm \ref{Algo:2} using Lagrange first-order interpolation for Example \ref{exp:CA} with $\sigma(x)=0.5\sin(2\pi x)$
 and the initial density \eqref{eq:LLDP} in the linear Landau damping problem along a single sample path: $T=5$ \textup{(}up\textup{)} and $T=10$ \textup{(}down\textup{)}.}\label{Tabletime}
\centering
 \begin{tabular}{|c|c|c|c|c|c|}
 \hline
$T/N$ & Non-adaptive algorithm &Algorithm \ref{Algo:2}& Ratio \\ \hline
 5/300 & 28.184 & 4.568 & 6.2 \\ \hline
 5/400 & 44.034 & 6.201 & 7.1 \\ \hline
 5/500 & 60.631 & 7.678 & 7.9 \\ \hline
 5/600 & 82.141 & 8.695 & 9.5 \\ \hline
 \hline
 $T/N$ & Non-adaptive algorithm &Algorithm \ref{Algo:2} &Ratio \\ \hline
 10/600 & 106.081 & 9.894 & 10.7 \\ \hline
 10/800 & 169.629 & 12.871 & 13.2 \\ \hline
 10/1000 & 234.888 & 16.283 & 14.4 \\ \hline
 10/1200 & 320.077 & 18.794 & 17.0 \\ \hline
 \end{tabular}
\end{table}
As mentioned in Remark \ref{rem:1}, a traditional approach is to truncate the velocity domain at the $n$th step into the domain $[-U_0 - nA_\tau, U_0 + nA_\tau]^d$, and then use the semi-Lagrangian method to solve the stochastic Vlasov equation \eqref{eq:Vla}. For convenience, we refer to this method as the non-adaptive algorithm. In Tab.\ \ref{Tabletime}, we compare the CPU time(s) of the non-adaptive algorithm and Algorithm \ref{Algo:2} for Example \ref{exp:CA} with $\sigma(x)=0.5\sin(2\pi x)$ and the initial density \eqref{eq:LLDP} in the linear Landau damping problem. Here we use the threshold $\epsilon_0=10^{-8}$ and four different time stepsizes $\tau = 1/60, 1/80, 1/100, 1/120$ for two terminal times $T= 5$ and $T=10$. The initial velocity domain is truncated into $[-6,6]$, and the phase space stepsizes are set to be $\delta x = \delta v = 1/200$. It can be observed that Algorithm \ref{Algo:2} significantly improves the efficiency of the non-adaptive algorithm for stochastic problems, especially for small time stepsize $\tau$ and large terminal time $T$. 
\begin{figure}[!htb]
 \centering
 \subfigure[$\sigma\equiv0$]{\includegraphics[width=1\textwidth]{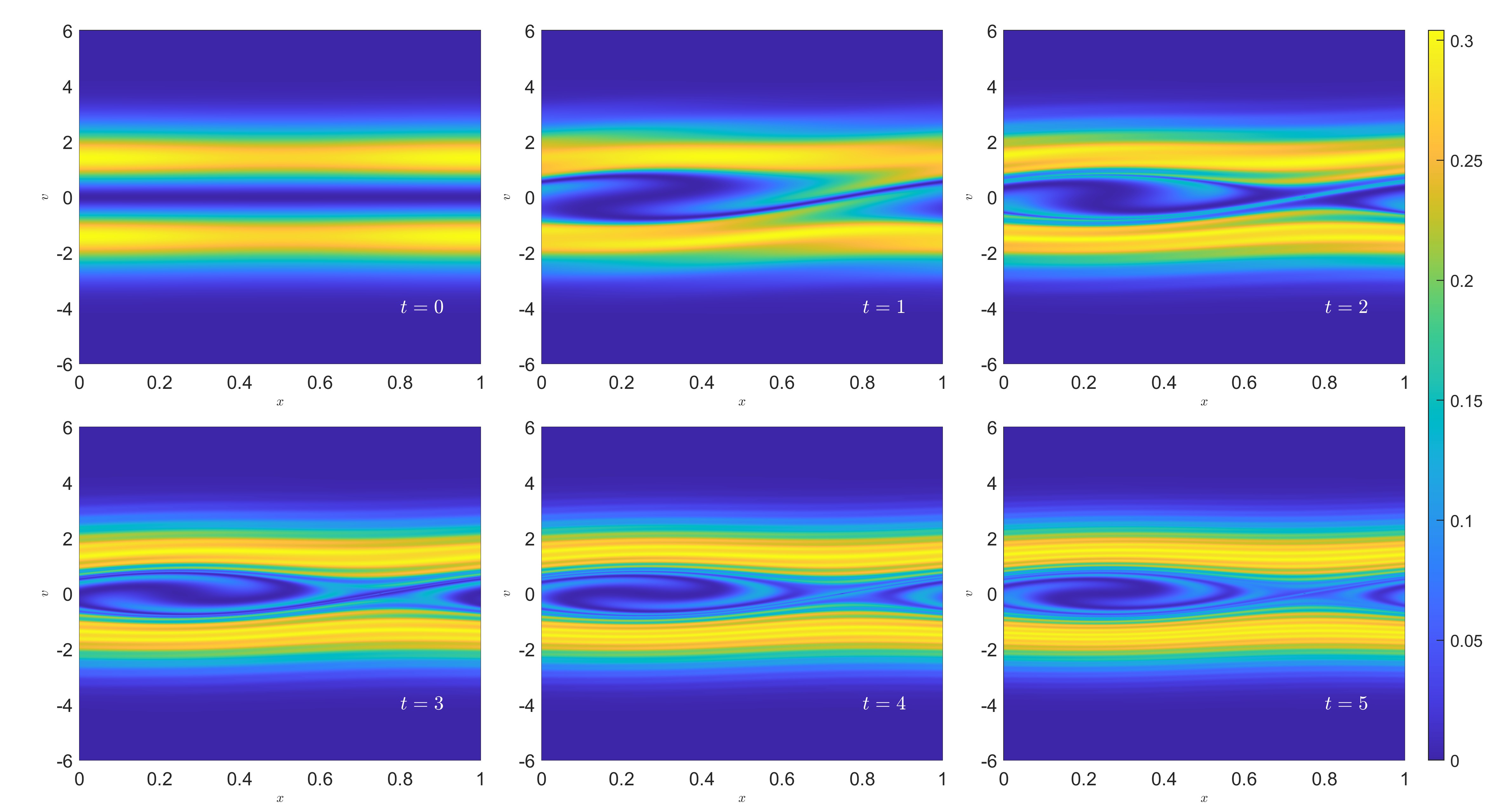}}
 \subfigure[$\sigma(x)=0.5(\cos(2\pi x)+1)$]{\includegraphics[width=1\textwidth]{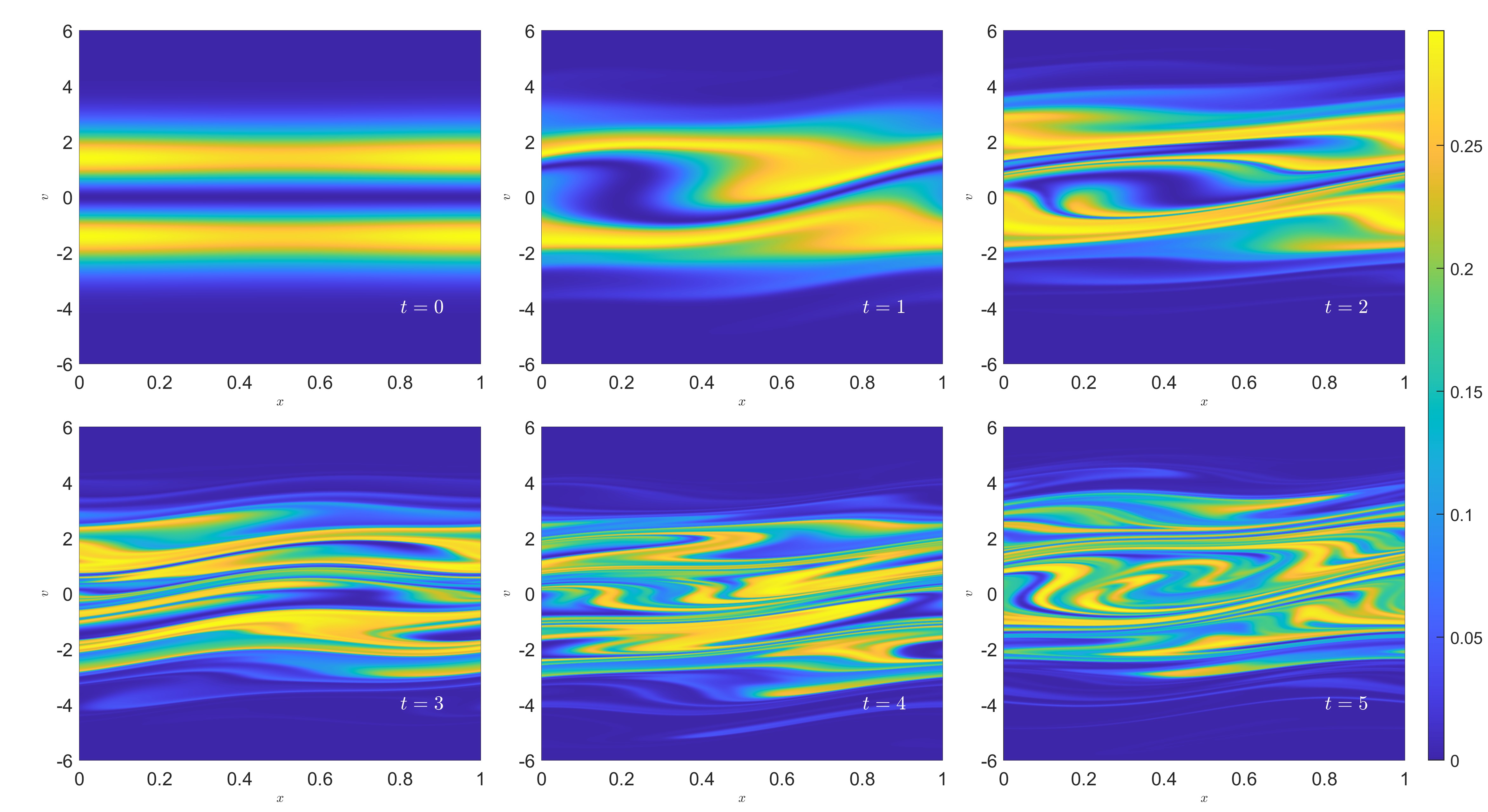}}
 \caption{Contour plot for Example \ref{exp:CA} with the initial density \eqref{eq:TSI} in the two stream instability problem at times $\{0,1,2,3,4,5\}$:  $\sigma\equiv0$ \textup{(}a\textup{)} and  $\sigma(x)=0.5(\cos(2\pi x)+1)$ \textup{(}b\textup{)}.
 }\label{FshapeTS-M}
\end{figure}

\subsection{Evolution of the solution}\label{S5.2}
In this subsection, we simulate the evolution of solutions of the stochastic linear Vlasov equation (i.e., Example \ref{exp:CA}) and the following stochastic Vlasov--Poisson equation to illustrate the impact of transport noise. In these tests, we adopt Algorithm \ref{Algo:2} with \eqref{eq:inv-SS} and the Lagrange first-order interpolation.
\begin{example}\label{exam:VP}
Consider the following stochastic Vlasov--Poisson equation
\begin{equation}\label{eq:VPtest}
\left\{
\begin{split}
	&\ud_tf +\left(v \partial_xf+E(t,x)\partial_vf\right)\ud t
	+\sigma(x)\partial_v f\odot \ud \beta(t)=0,\\
&\partial_{x} E(t,x)=\int_{\R}f(t,x,v)\ud v-1,
\end{split}
\right. \quad t\in[0,T]
\end{equation}
in $\mathbb{T}\times\R$, 
where $f_0$ is given by \eqref{eq:LLDP} or \eqref{eq:TSI} with $\alpha=0.01$ and $L=4\pi$. Here $\sigma:\mathbb{T}\to\R$ is globally Lipschitz continuous and $\{\beta(t)\}_{t\in[0,T]}$ is a standard 1-dimensional Brownian motion.
\end{example}

 \begin{figure}[htb]
 \centering
 \subfigure[$\sigma\equiv 0$]{\includegraphics[width=1\textwidth]{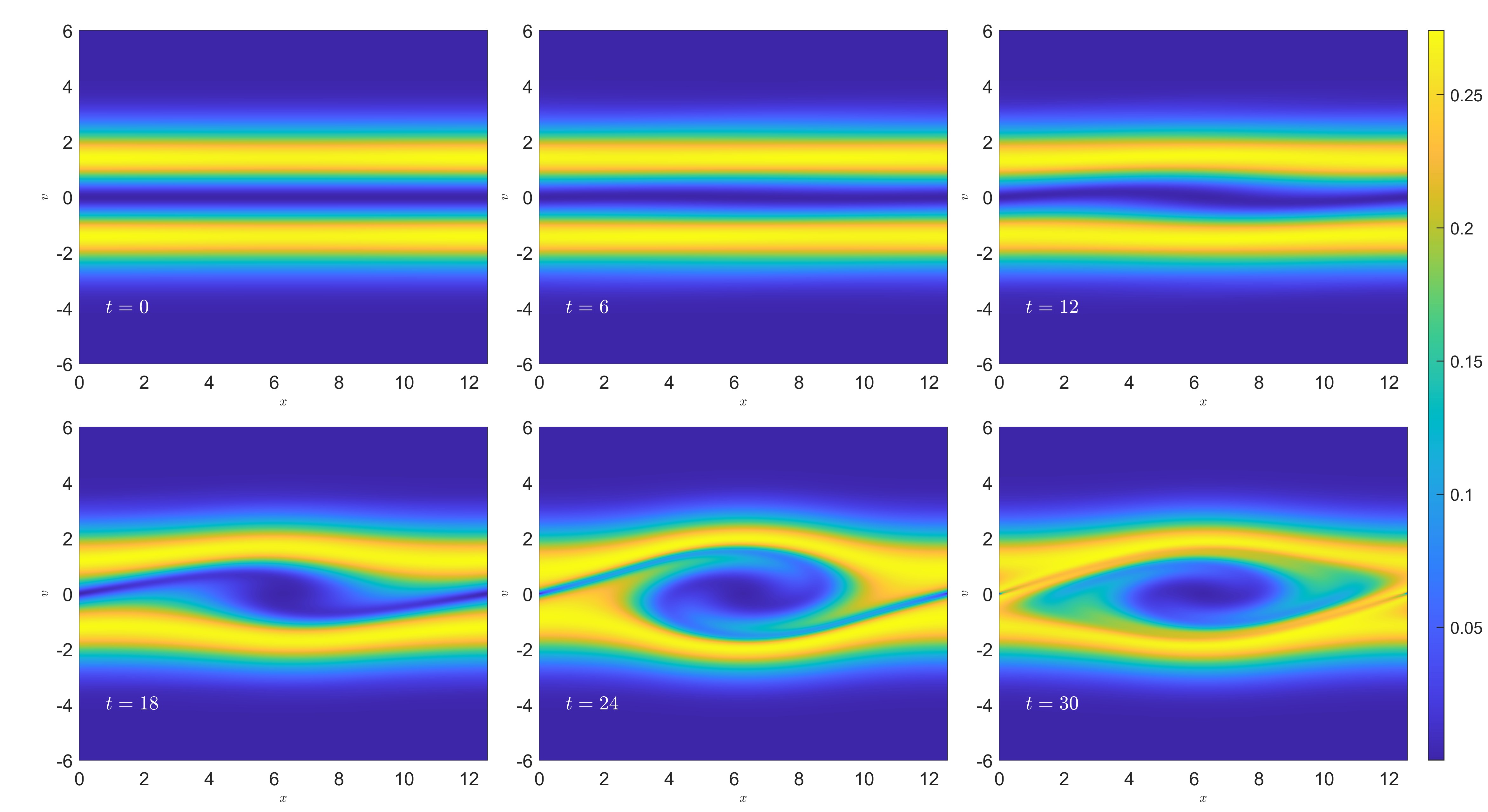}}
 \subfigure[$\sigma\equiv 0.5$]{\includegraphics[width=1\textwidth]{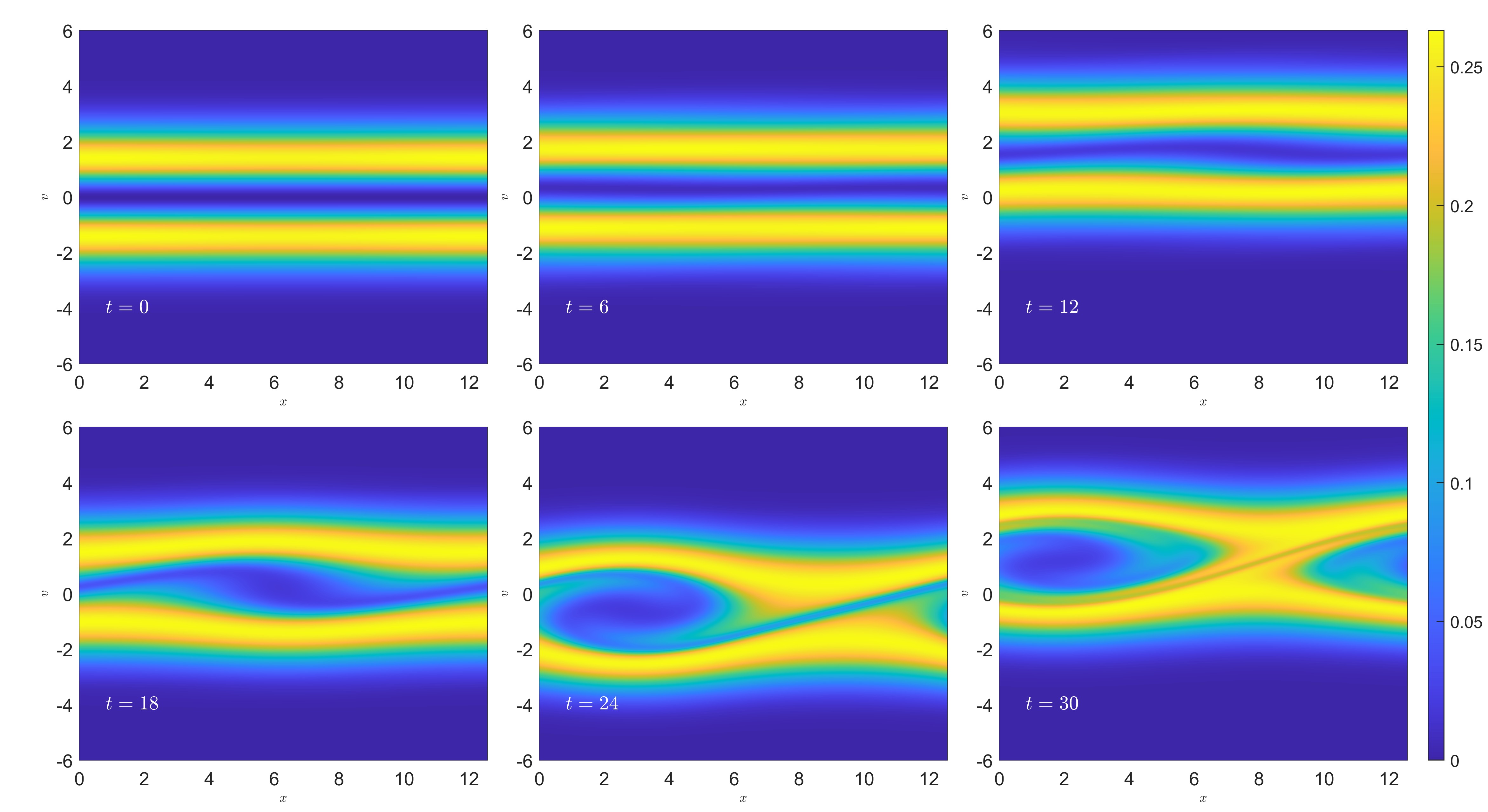}}
 \caption{Contour plot for Example \ref{exam:VP} with the initial density \eqref{eq:TSI} in the two stream instability problem at times $\{0,6,12,18,24,30\}$: $\sigma\equiv 0$ \textup{(}a\textup{)} and $\sigma\equiv 0.5$ \textup{(}b\textup{)}.}\label{VPshapeSTr}
\end{figure}

In Fig.\ \ref{FshapeTS-M}, snapshots of the numerical solution
 for Example \ref{exp:CA} with the initial density \eqref{eq:TSI} at $t= 0,1,2,3,4,5$ are displayed, with $\sigma\equiv 0$ (deterministic case) and $\sigma(x)=0.5(\cos(2\pi x)+1)$ (stochastic case).
 The parameters in Algorithm \ref{Algo:2} are $\tau=5/300$, $\delta x=\delta v=1/200$ and $\epsilon_0=f_0(0,U_0)$ with $U_0=6$. 
 It can be observed from Fig.\ \ref{FshapeTS-M}(a) that the vortices in the deterministic case remain stable and persist over time (see also \cite[Fig.\ 1]{BC24}).
In contrast, in the stochastic case, the vortices fail to fully develop and tend to dissipate due to the effect of transport noise, as shown in Fig.\ \ref{FshapeTS-M}(b). One can also observe that
the numerical solution in Fig.\ \ref{FshapeTS-M} remains non-negative, demonstrating the positivity-preserving property of Algorithm \ref{Algo:2} with the Lagrange first-order interpolation.

In Fig.~\ref{VPshapeSTr},
snapshots of the numerical solution for Example \ref{exam:VP} with the initial density \eqref{eq:TSI} at $t= 0,6,12,18,24,30$ are presented, 
 where we observe the vortex structure in both the deterministic case ($\sigma\equiv 0$) and stochastic case ($\sigma\equiv 0.5$). Here we take $\tau = 0.1$, $ \delta x = \delta v = 4\pi/256$, and $\epsilon_0=f_0(0,U_0)$ with $U_0=123\pi/64$.
 Fig.~\ref{VPshapeSTr}(a) 
 suggests that the instability grows rapidly and a hole structure appears from time $t=12$ to $t=18$. After $t=24$, the trapped particles oscillate within the electrostatic potential, and the vortex undergoes periodic rotation. Compared to the deterministic case, there are stochastic translations in both $x$ and $v$ directions for the vortex, as shown in Fig.\ \ref{VPshapeSTr}(b).
 This verifies that the solution of \eqref{eq:VPtest} with $\sigma$ being a constant is equivalent to that of the deterministic Vlasov–Poisson equation with a stochastic coordinate transformation (see Remark \ref{rem1}).

\subsection{Preservation of integrals}\label{S5.3}
In this subsection, we compare the traditional semi-Lagrangian method and Algorithm \ref{Algo:2} with \eqref{eq:inv-SS} and the Lagrange first-order interpolation to illustrate the necessity of enlarging the computational domain in phase space for the stochastic problem \eqref{eq:Vla}.
Meanwhile, we show that 
 Algorithm \ref{Algo:2} with \eqref{eq:inv-SS} exhibits a better performance in perserving the integrals, compared to that with the Euler--Maruyama method.
 
\begin{figure}[!htbp]
	\centering
	\subfigure[]{\includegraphics[width=0.35\textwidth]{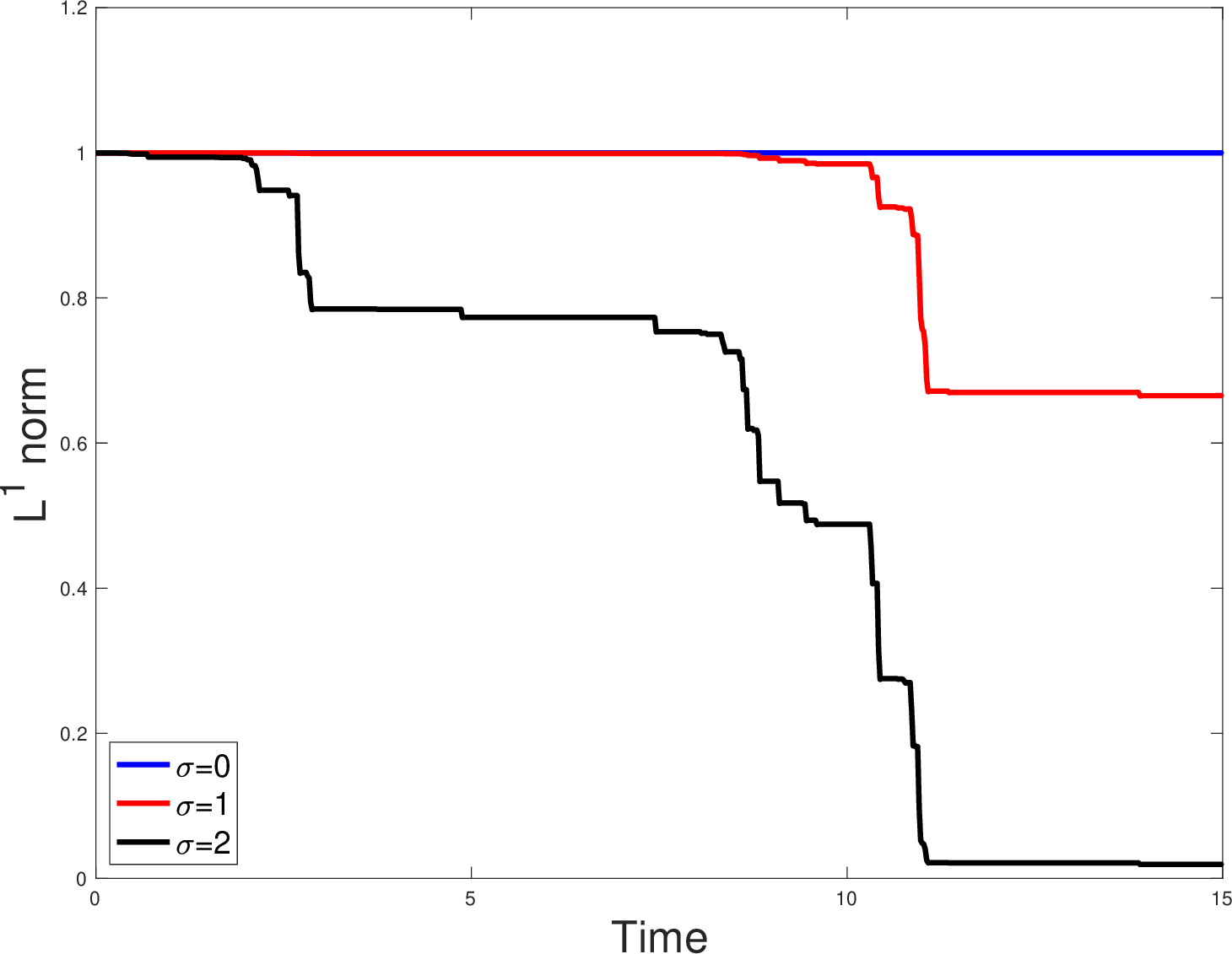}}\hspace{4em}
	\subfigure[]{\includegraphics[width=0.35\textwidth]{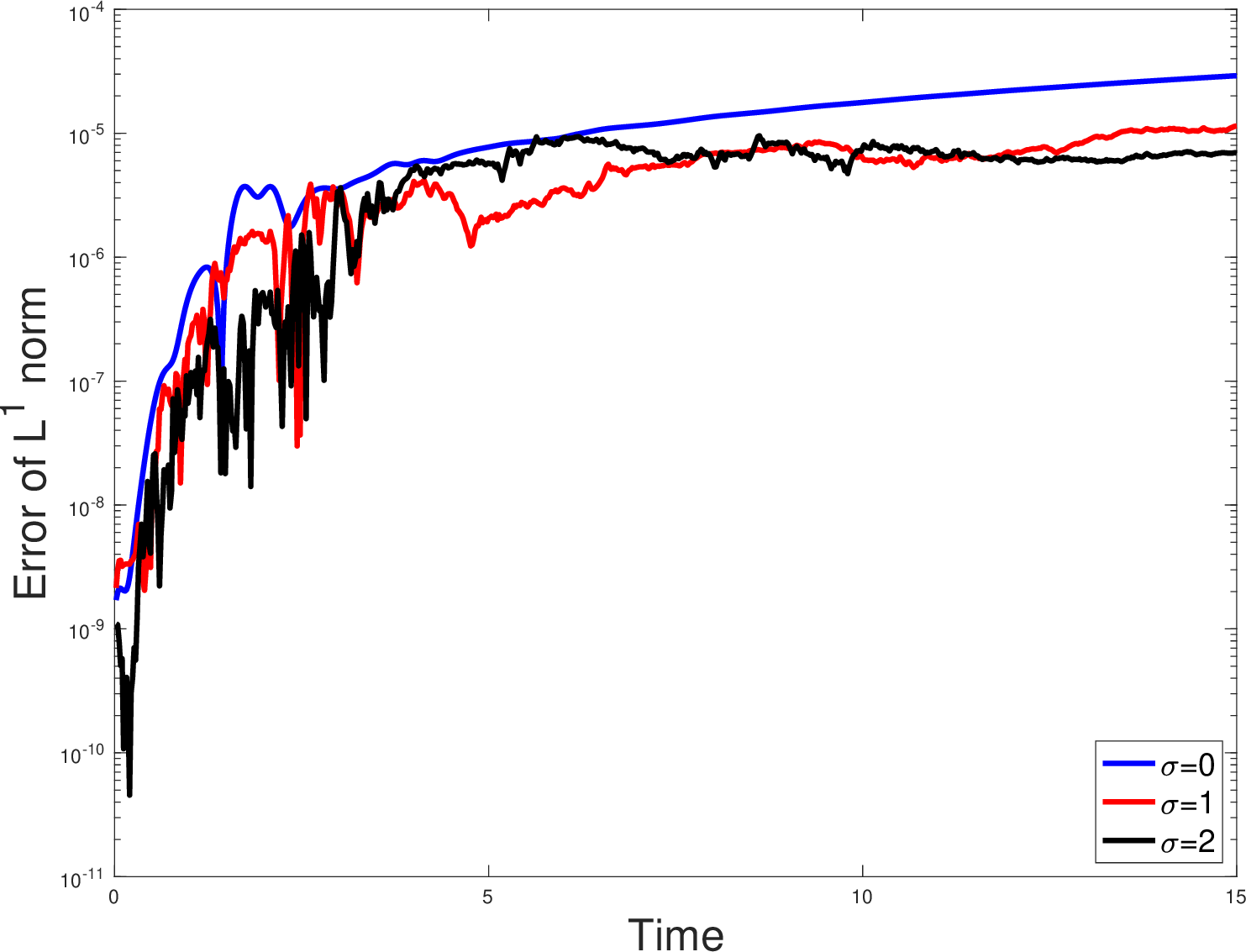}}
	\caption{\textup{(}a\textup{)} Evolution of $L^1(\mathbb T\times[-6,6])$-norm of the numerical solution of the traditional semi-Lagrangian method and \textup{(}b\textup{)} Evolution of mass error 
	of Algorithm \ref{Algo:2}, 	
	for Example \ref{exp:CA} with different $\sigma$ and the initial density \eqref{eq:TSI} in the two stream instability problem along a single sample path.
	}\label{Fig:L1} 
\end{figure}
Fig.\ \ref{Fig:L1} presents the evolution of the $L^1$-norm of the numerical solution for Example \ref{exp:CA} with different noise intensities ($\sigma\equiv0,1,2$), where $T=15$, $f_0$ is given by \eqref{eq:TSI}, and the stepsizes are $\tau=0.02$ and $\delta x=\delta v=0.01$.
In Fig.\ \ref{Fig:L1}(a), we truncate the phase space domain into $\mathbb{T} \times [-6, 6]$ and then compute the numerical solution using the traditional semi-Lagrangian method, with the value of the numerical solution outside the truncated domain being enforced to zero.
 It can be seen that the $L^1(\mathbb T\times[-6,6])$-norm of the numerical solution remains nearly invariant for $\sigma\equiv0$. However, for $\sigma(x)>0$, the $L^1(\mathbb T\times[-6,6])$-norm of the numerical solution decreases over time, which 
emphasizes the necessity of enlarging the computational domain in phase space for the stochastic problem \eqref{eq:Vla}.
In comparison, we compute the mass for Example \ref{exp:CA} by
Algorithm \ref{Algo:2} with \eqref{eq:inv-SS}.
 By taking $\epsilon_0=f_0(0,U_0)$ and $U_0=6$, we plot in
Fig.\ \ref{Fig:L1}(b) the error $|\|f^{\delta y,\tau,\epsilon_0}_n\|_{L^1(\mathbb T\times[-U_n,U_n])}-\|f^{\delta y,\tau,\epsilon_0}_0\|_{L^1(\mathbb T\times[-U_0,U_0])}|$ against time, which verifies that the mass of the numerical solution is nearly invariant over time.

 \begin{figure}[!htbp]
	\centering
	\subfigure[\footnotesize{$L^1(\mathbb{T}\times[-U_n,U_n])$-norms}]{\includegraphics[width=0.35\textwidth]{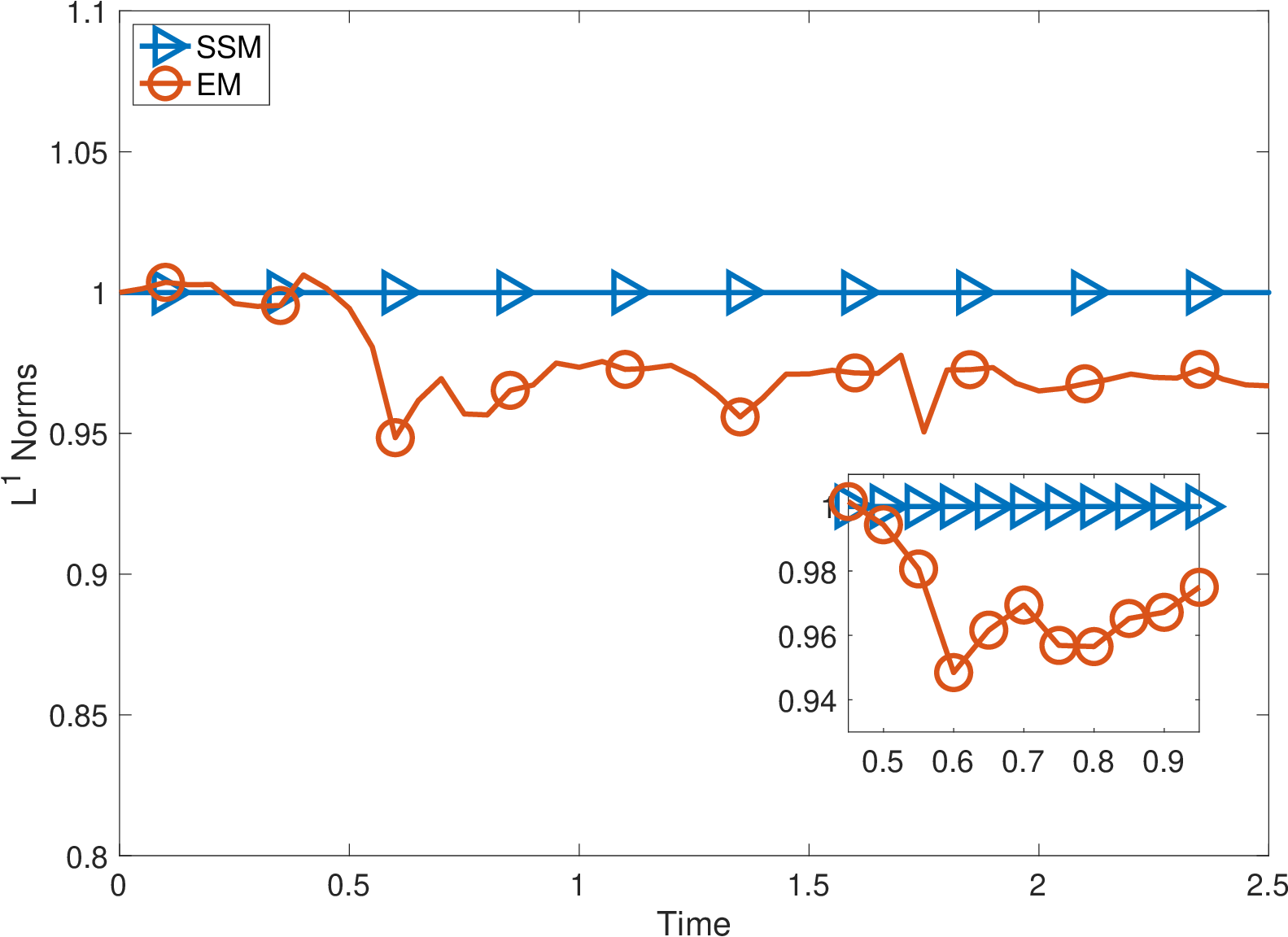}}\hspace{4em}
	\subfigure[\footnotesize{$L^2(\mathbb{T}\times[-U_n,U_n])$-norms}]{\includegraphics[width=0.35\textwidth]{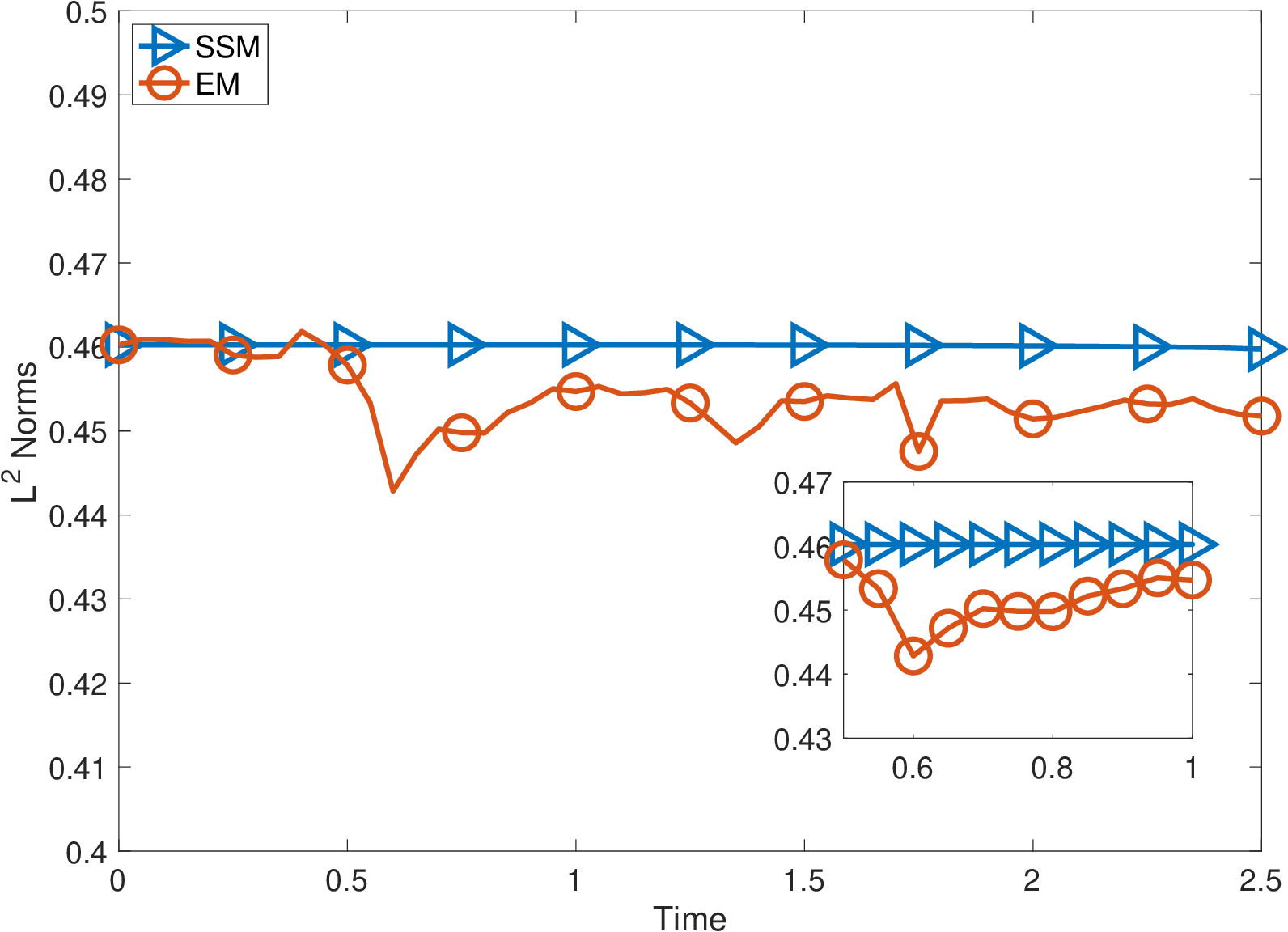}}
	 \caption{Evolution of $L^p(\mathbb{T}\times[-U_n,U_n])$-norms of numerical solutions of Algorithm \ref{Algo:2} with \eqref{eq:inv-SS} and the Euler--Maruyama method
	  for Example \ref{exp:CA} with $\sigma(x)=\sin(2\pi x)$ and the initial density \eqref{eq:TSI} in the two stream instability problem along a single sample path: $p=1$ \textup{(}a\textup{)} and $p=2$ \textup{(}b\textup{)}.
}
	\label{Fig:Norm}
\end{figure}
We 
compute
 the $L^p(\mathbb T\times[-U_n,U_n])$-norms ($p=1,2$) of the numerical solution $f_n^{\delta y,\tau,\epsilon_0}$ with $U_0=6$ and 
  $\tau=0.05$. 
The results are displayed in Fig.\ \ref{Fig:Norm} for Example \ref{exp:CA} with $\sigma(x)=\sin(2\pi x)$ and the initial density 
\eqref{eq:TSI}.
To highlight the error stemming from time discretization, relatively small phase space stepsizes $\delta x=\delta v=1/4000$ and threshold $\epsilon_0=10^{-6}$ are utilized. As displayed in Fig.\ \ref{Fig:Norm}, compared to the Euler--Maruyama method (which is not volume-preserving),  \eqref{eq:inv-SS} exhibits a better performance in preserving $L^p$-norms of the numerical solution.

\subsection{Evolution of physical quantities}\label{S5.4}
In this subsection, we plot the evolution of physical quantities
by Algorithm \ref{Algo:2} with the spline interpolation and \eqref{eq:inv-SS} to
verify the results in Proposition \ref{tho:Hft}. We note that the spline interpolation offers higher computational accuracy compared to the Lagrange first-order interpolation, but it does not preserve positivity (see, e.g., \cite{BM08}).
\begin{figure}[!htbp]
\centering
\subfigure[]
 { \includegraphics[width=0.32\textwidth]{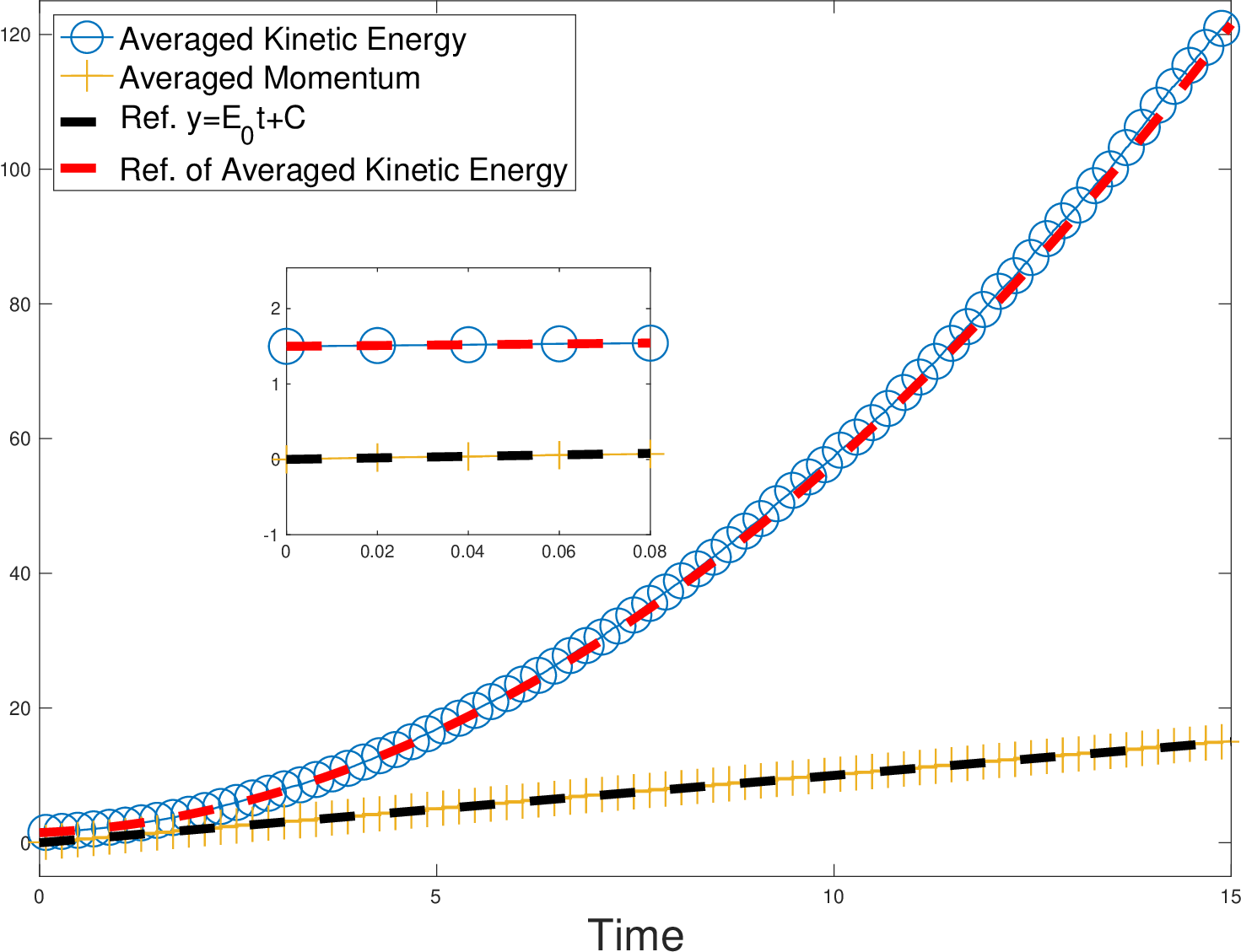}}
 \subfigure[]
 {\includegraphics[width=0.32\textwidth]{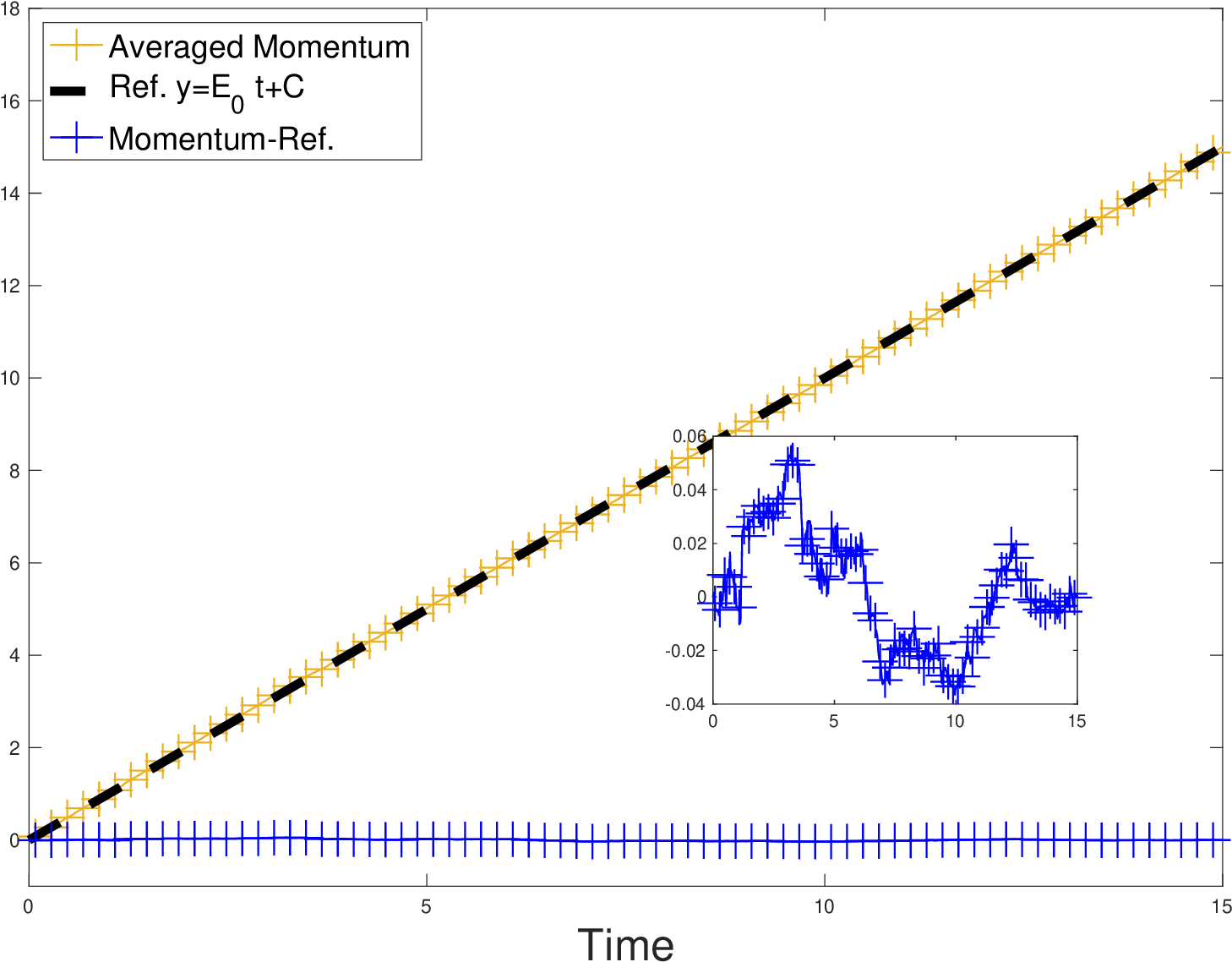}}
\subfigure[]
 {\includegraphics[width=0.32\textwidth]{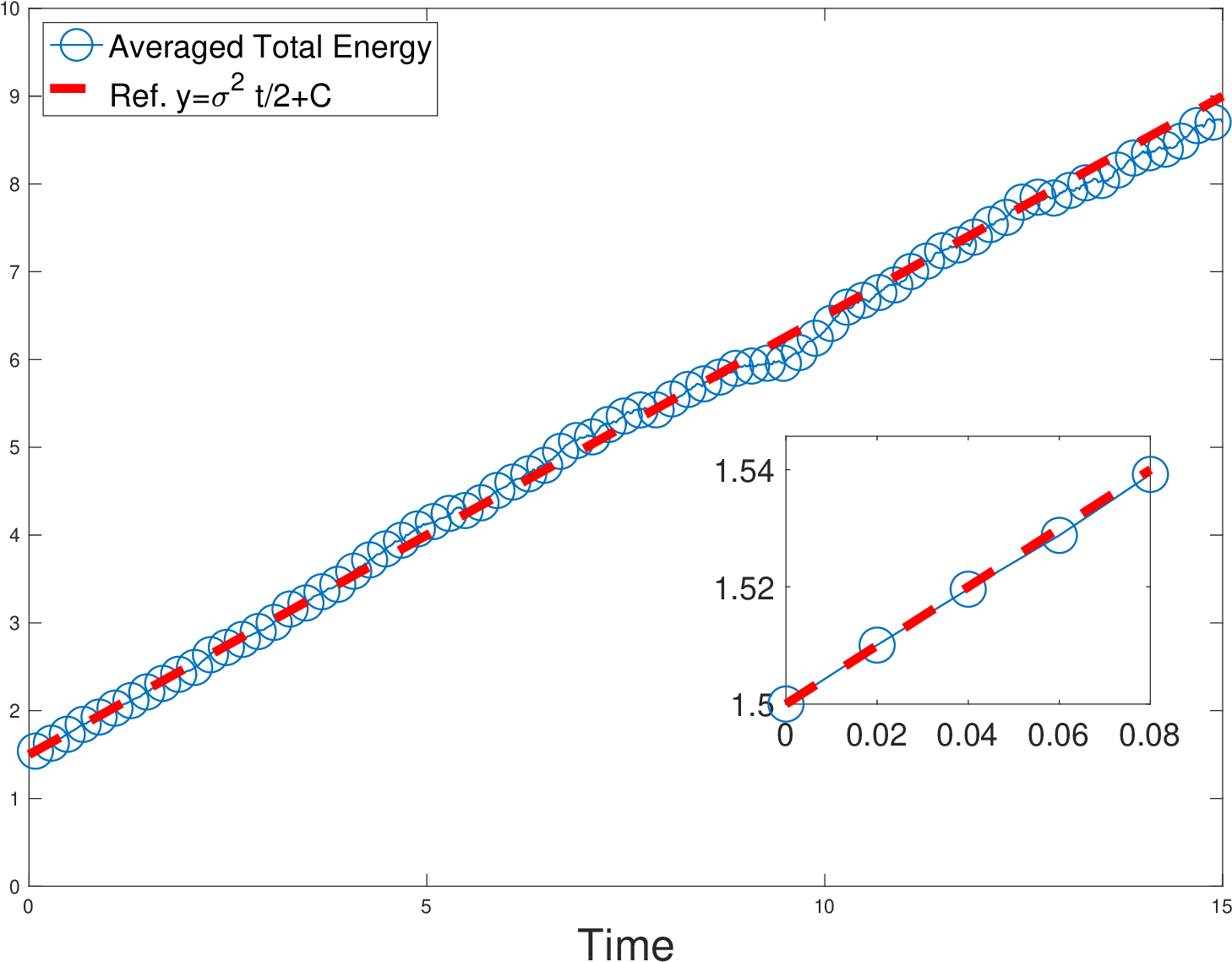}}
 \caption{Evolution of physical quantities for Example \ref{exp:CA} with the initial density \eqref{eq:TSI} in the two stream instability problem: $E\equiv E_0=1$ and $\sigma\equiv 1$ \textup{(}a\textup{)}, $E\equiv E_0=1$ and $\sigma(x)=0.5(\cos(2\pi x)+1)$ \textup{(}b\textup{)}, and $E(x)=\cos(2\pi x)$ and $\sigma \equiv 1$ \textup{(}c\textup{)}.}\label{Fig:EPC} 
\end{figure} 

\begin{figure}[!htbp]
\centering
 \subfigure[]
 {\includegraphics[width=0.4\textwidth]{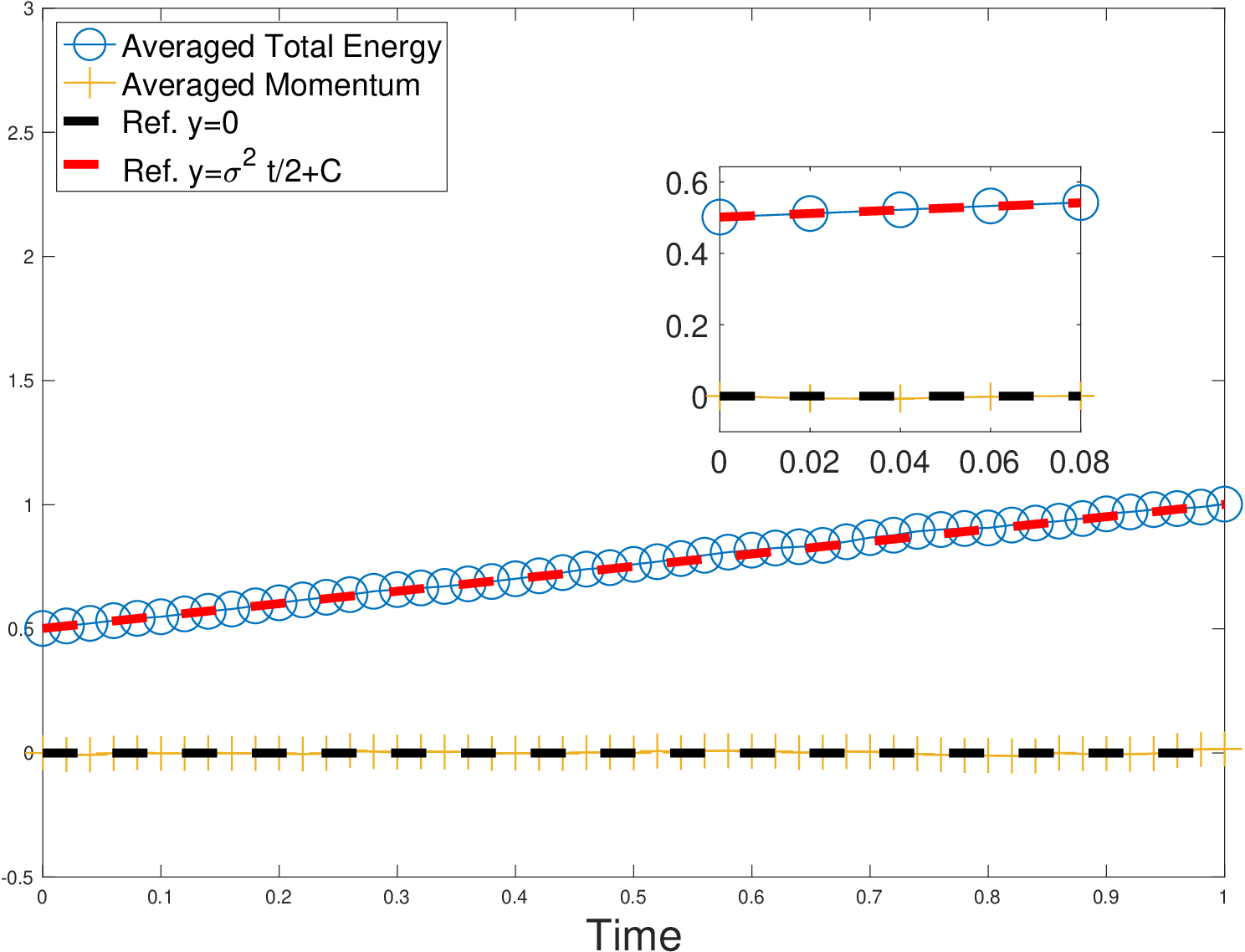}}\hspace{4em}
 \subfigure[]
 {\includegraphics[width=0.4\textwidth]{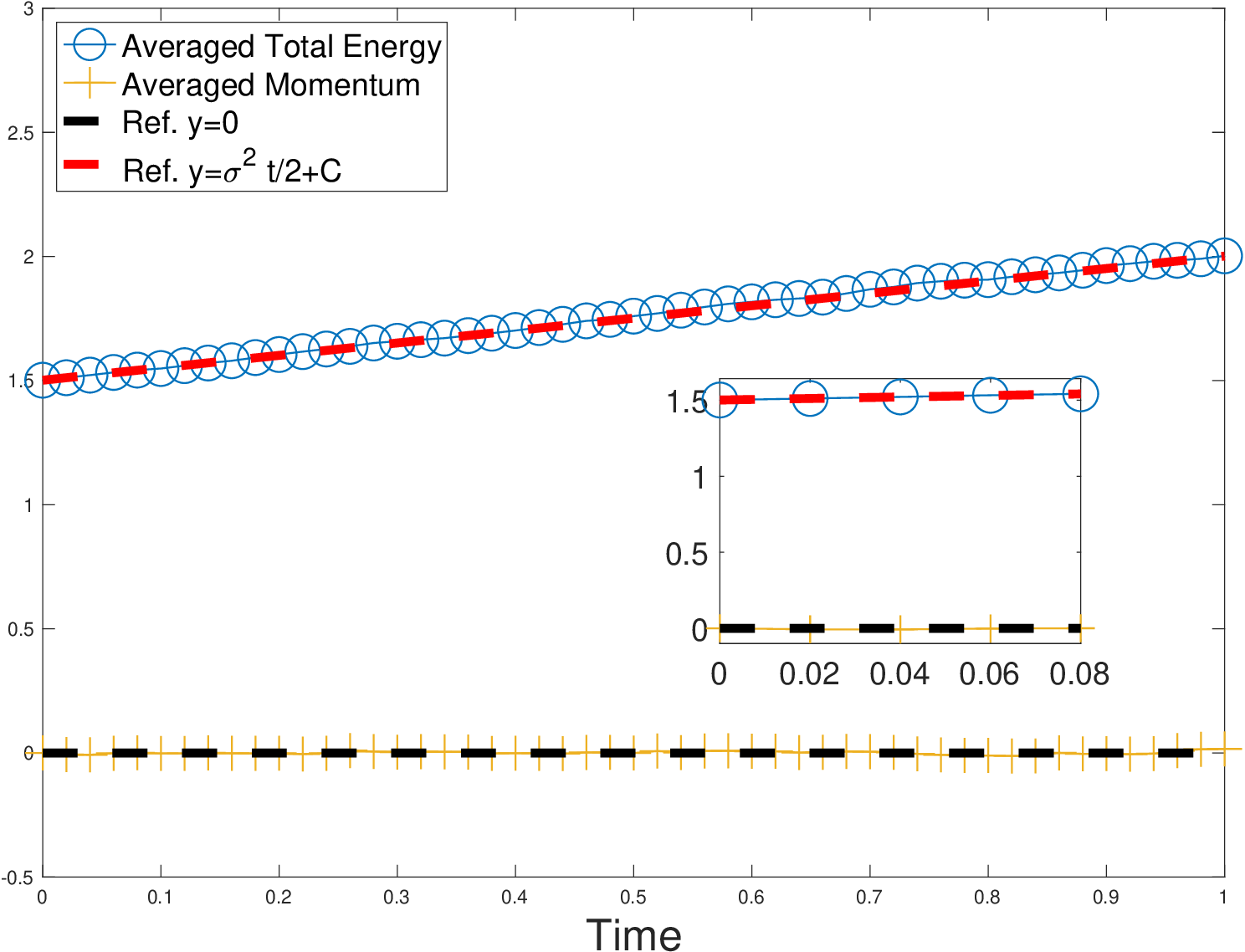}} \caption{Evolution of physical quantities for Example \ref{exam:VP} with $\sigma\equiv1$: the initial density \eqref{eq:LLDP} in the linear Landau damping problem \textup{(}a\textup{)} and the initial density \eqref{eq:TSI} 
 in the two stream instability problem \textup{(}b\textup{)}. 
}\label{Fig:EPCVP}
\end{figure}

Figs.\ \ref{Fig:EPC}-\ref{Fig:EPCVP} display the evolution of the averaged momentum $\E[\mathcal P[f_n^{\delta y,\tau,\epsilon_0}]]$, averaged kinetic energy $\E[\mathcal K[f_n^{\delta y,\tau,\epsilon_0}]]$, and averaged
total energy $\E[\mathcal H[f_n^{\delta y,\tau,\epsilon_0}]]$. 
 The expectations are realized by the average of 1000 sample paths.
 In these tests, we take
$U_0=6$, $\epsilon_0=10^{-6}$, $\tau=0.02$, and $\delta x=\delta v=0.01$. We observe from Fig.\ \ref{Fig:EPC} that
\begin{itemize}%[leftmargin= 1em, itemindent=2em]
\item[(1)] When $E$ and $\sigma$ are constants, the averaged kinetic energy grows nearly quadratically over time (see Fig.\ \ref{Fig:EPC}(a)), verifying \eqref{eq:Kenergy} in Proposition \ref{tho:Hft}(i);
\item[(2)] When $E$ is a constant, the averaged momentum exhibits nearly linear growth with respect to time (see Fig.\ \ref{Fig:EPC}(a)-(b)), validating \eqref{eq:momentum} in Proposition \ref{tho:Hft}(i);
\item[(3)] When $E=-\nabla u$ with $u(x)=\sin(2\pi x)$ and $\sigma$ is a constant, Fig.\ \ref{Fig:EPC}(c) verifies that random forces 
cause the averaged total energy to increase linearly over time, which is consistent with Proposition \ref{tho:Hft}(ii). 
\end{itemize}
Fig.\ \ref{Fig:EPCVP} shows that when $\sigma$ is a constant, the averaged total energy of the stochastic Vlasov--Poisson equation grows linearly with respect to time with the slope $\sigma^2/2$, while the averaged momentum remains invariant.
This coincides with Proposition \ref{tho:Hft}(iii).

\section*{Acknowledgments}
We would like to thank Prof. Arnulf Jentzen (CUHK-Shenzhen, University of Münster) and Dr. Yingzhe Li (Max Planck Institute for Plasma Physics)
 for their valuable suggestions and discussions.

\bibliographystyle{plain}
\bibliography{MAVF}
\end{document}